\documentclass[pdflatex,sn-mathphys-num]{sn-jnl}

\usepackage{graphicx}%
\usepackage{multirow}%
\usepackage{amsmath,amssymb,amsfonts}%
\usepackage{amsthm}%
\usepackage{mathrsfs}%
\usepackage[title]{appendix}%
\usepackage{xcolor}%
\usepackage{textcomp}%
\usepackage{manyfoot}%
\usepackage{booktabs}%
\usepackage{algorithm}%
\usepackage{algorithmicx}%
\usepackage{algpseudocode}%
\usepackage{listings}%
\numberwithin{equation}{section}%
\usepackage{amstext}%
\usepackage{enumitem}%
\usepackage{diagbox}%
\usepackage{comment}%
\def\R{\mathbb{R}}
\def\N{\mathbb{N}}
\def\bR{\bar{\mathbb{R}}}
\def\dom{\operatorname{dom}}

\newcommand{\rank}[1]{\textrm{rank}(#1)}
\newcommand{\norm}[1]{\| #1 \|}
\newcommand{\yck}{{y_{reg}^k}}
\newcommand{\Xck}{{X_{reg}^k}}
\newcommand{\Yck}{{Y_{reg}^k}}
\newcommand{\setRank}{{\mathcal{C}_r}}
\newcommand{\setNorm}{{\mathcal{C} }}
\newcommand{\proj}{\mathcal{P}}
\newcommand{\projRank}{{\mathcal{P}_{\mathcal{C}_r}}}
\newcommand{\projNorm}{{\mathcal{P}_{\mathcal{C} }}}

\newcommand{\cZ}{G}
\newcommand{\ML}{G}
\newcommand{\prox}{\operatorname{prox}}
\newcommand{\iif}{\Leftrightarrow}
\def\diag{\text{diag}}
\newcommand{\Tsvd}[1]{[ #1 ]_1}
\newcommand{\tsvd}[1]{[ #1 ]_2}
\newcommand{\REV}[1]{\textcolor{black}{#1}}
\newcommand{\ste}[1]{\textcolor{black}{#1}}
\theoremstyle{thmstyleone}%
\newtheorem{theorem}{Theorem}[section]% meant for sectionwise numbers
\newtheorem{lemma}[theorem]{Lemma}%
\newtheorem{corollary}[theorem]{Corollary}%

\theoremstyle{thmstyletwo}%
\newtheorem{remark}{Remark}[section]%

\theoremstyle{thmstylethree}%
\newtheorem{definition}{Definition}[section]%

\begin{document}

\title[Inexact AP Method]{An inexact alternating projection method with application to  matrix completion}


\author*[1,4]{\fnm{Stefania} \sur{Bellavia}}\email{stefania.bellavia@unifi.it}

\author[2,4]{\fnm{Simone} \sur{Rebegoldi}}\email{simone.rebegoldi@unimore.it}

\author[3,4]{\fnm{Mattia} \sur{Silei}}\email{mattia.silei@unifi.it}

\affil*[1]{\orgdiv{Dipartimento di Ingegneria Industriale (DIEF)}, \orgname{Universit\`a degli Studi di Firenze}, \orgaddress{\street{Viale G.B. Morgagni 40}, \city{Firenze}, \postcode{50134}, \state{Italia}}}

\affil[2]{\orgdiv{Dipartimento di Scienze Fisiche, Informatiche e Matematiche}, \orgname{Universit\`a degli Studi di Modena e Reggio Emilia}, \orgaddress{\street{Via Campi 213/B}, \city{Modena}, \postcode{41125}, \state{Italia}}}

\affil[3]{\orgdiv{Dipartimento di Matematica e Informatica (DIMAI)}, \orgname{Universit\`a degli Studi di Firenze}, \orgaddress{\street{viale G.B. Morgagni 65}, \city{Firenze}, \postcode{50134}, \state{Italia}}}

\affil[4]{\orgdiv{INdAM Research Group GNCS}}


\abstract{
We develop and analyze  an inexact regularized alternating projection method for nonconvex feasibility problems. Such a method employs inexact projections on one of the two sets, according to a set of well-defined conditions. We prove the global convergence of the algorithm, provided that a certain merit function satisfies the Kurdyka-\L{}ojasiewicz property on its domain.
The method is then specialized to the class of affine rank minimization problems, which includes matrix completion as a special case.
We approximate the truncated  Singular Value Decomposition of the matrix that has to be projected by means of a Krylov solver, and provide suitable  stopping criteria for the Krylov method complying with the theoretical inexactness conditions. The information needed to implement such stopping criteria do not require an extra computational effort as they 
are a by-product of the Krylov method itself and avoid the so called oversolving phenomena. 
Results of the numerical validation of the algorithm on matrix completion problems are presented.
}




\maketitle

%%%%%%%%%%%%%%%%
% Introduction %
%%%%%%%%%%%%%%%%
\section{Introduction}

This paper is concerned with regularized \REV{alternating} projection methods for computing  a point in the intersection of two nonconvex sets $A$ and $B$, namely
\begin{equation}\label{eq:problem}
    \text{find }x\text{ such that }x\in A \cap B\quad\quad A,B \subset \R^n.
\end{equation}
The alternating projection method (APM) dates back to the seminal paper \cite{Neumann-etal-1950} by J. von Neumann in 1950. Starting from a point in the first set, APM generates a sequence of points by alternately projecting onto the two sets. The method converges globally when $A,B$ are convex \cite{bauschke1996projection}, and a rich literature has been devoted to the convergence analysis of APM in the convex case \cite{bauschke1993, bauschke1994, bauschke1996projection, escalanteraydan, santos2021}. When one of the sets is nonconvex, APM converges only locally and further hypotheses on the sets and their intersection are required \cite{drusvyatskiyPhDthesis2013,drusvyatskiy2015transversality,lewis2008alternating,lewis2009local,Noll-etal-2016}.

\subsection{Contribution of the paper}
We focus on a regularized variant of APM (RAPM) proposed in \cite{attouch2010proximal}, \REV{for which }
global convergence \REV{holds} even in the case of nonconvex sets. \REV{RAPM requires exact projections on both sets; hence, it is not applicable when} the projection on one of the sets \REV{is exact}, while it is not possible or too expensive to compute the exact projection onto the other set.  We design a theoretically well-founded inexact variant of RAPM, denoted by iRAPM, that allows to use inexact projections on the latter set. Inspired by \cite{Birgin2003,Ferreira-2022,Bonettini-Ochs-Prato-Rebegoldi-2023}, we introduce suitable conditions  that the inexact projection has to satisfy in order to preserve the theoretical properties of the exact counterpart RAPM. 

Our main motivation for this study lies in the application of alternating projection methods to Matrix Completion (MC) problems and, \REV{more generally,} affine rank minimization problems, which consist in finding a matrix $Z\in \R^{n_1\times n_2}$ of minimum rank that satisfies a given system of linear equations:
 \begin{equation}
    \label{affine_rank_min}
    \begin{aligned}
        \min \quad & \textrm{rank}(Z), \\    
        \textrm{s.t. } {\cal A}(Z) = b,
    \end{aligned} 
\end{equation}
where ${\cal A}: \R^{n_1\times n_2}\rightarrow \R^{q} $ is a linear map and $b \in \R^q$. 
Problem \eqref{affine_rank_min} encompasses the MC problem, system identification and control problems, Euclidean embedding, and collaborative filtering \cite{Fazel_review}.
 
Assuming to know the sought rank of the matrix $Z$, say $r$, we reformulate the affine rank minimization problem as the problem of finding a matrix in the intersection of two sets: the set \REV{$\mathcal{C}$} of matrices satisfying the linear constraints and the \textit{rank level set}, that is, the nonconvex set \REV{$\mathcal{C}_r$} of matrices of rank less than or equal to $r$. Then, we specialize \REV{APM, RAPM and} iRAPM to this latter class of problems. First, we show that the two sets \REV{$\mathcal{C}$ and $\mathcal{C}_r$} satisfy the assumptions that guarantee \REV{the convergence of the three methods}, thus providing a theoretical support to the employment of APMs for solving  affine rank minimization problems. Second, we focus on the computation of the projection onto the rank level set, namely the truncated Singular Value Decomposition (SVD) of the matrix that has to be projected. This is carried out by an iterative method, with an inherent truncation error. In the large scale setting, this is an expensive task and iterative Krylov methods are used \cite{saad2011numerical} in order to avoid \REV{full} matrix factorizations. Suitable stopping criteria for this inner iterative process are therefore of crucial importance in order to save Krylov inner iterations, while preserving the same convergence properties as RAPM. \REV{We propose novel stopping criteria for the Krylov method that comply with the theoretical conditions imposed in iRAPM.} The information needed to implement such stopping criteria does not require  extra computational effort as it is a by-product of the Krylov method itself. 

Some numerical experiments on matrix completion problems to validate the convergence results, the proposed stopping criteria and the performance of the studied methods will be also shown.

\subsection{Related works}
Early general results on inexact APM schemes in \cite{kruger} emerge naturally from the convergence theory of the exact algorithm, but their application to common feasibility problems is not investigated. In the works \cite{ferreira_inexact_2021,drusvyatskiy2019local}, inexact APMs are proposed for problems where, similarly to our setting, the projection onto one of the two sets can be computed exactly while the other projection is inexact. On the one hand, in \cite{ferreira_inexact_2021}, the authors propose an APM where the projection onto one of the sets is approximately computed by means of the conditional gradient method.
On the other hand, in \cite{drusvyatskiy2019local}, inexact projections are considered onto two possible sets; the set of the solution of mixed equations and inequalities and a certain smooth manifold given by the image under a nonlinear map of an open set containing the iterates. Inexact projections are also considered in \cite{luke2013prox} where the author proposes a theoretical inexactness criterion based on normal cones.
We also highlight that both \cite{drusvyatskiy2019local} and \cite{luke2013prox} consider easy to project onto the rank level set since the projection of a matrix $Z$  is given by its  truncated SVD.
Our point of view is different, as the computation of the SVD is expensive and inherently inexact, and differently from \cite{luke2013prox}, we provide inexactness conditions and corresponding stopping criteria for the inner Krylov solver that can be implemented without any additional cost, allowing for the reduction of computational times. Such criteria are inspired by the works \cite{Bonettini-Prato-Rebegoldi-2020,Bonettini-Prato-Rebegoldi-2021,Bonettini-Ochs-Prato-Rebegoldi-2023}, where implementable inexactness conditions for computing proximal maps are considered in the context of inexact forward--backward algorithms. Unlike these previous works, we cannot rely on the convexity of the sets involved in problem \eqref{eq:problem}; for that reason, our definition of inexact projection will be slightly more restrictive.

The APM is widely applied to the Matrix Completion problem under different names. For example, the Expectation-Maximization method is used in statistics for parameters estimation with missing data, and,  when data have a matrix structure, some of its variants are equivalent to APM \cite{guerreiro2003estimation}. In the context of denoising and recovering missing data in time series, the Cadzow's algorithm is well studied \cite{wang2021fast}. It consists in wrapping the time series in a Hankel matrix, applying an SVD truncation to denoise the data, and imputing the missing values.  In its basic form, it consists in alternating the same two steps of APM. \REV{Note that the convergence results for RAPM when applied to the rank affine minimization problem provide also a theoretical support to regularized variants of the above approaches.}

\subsection{Outline}
The paper is organised as follows.
After an introduction to APM and RAPM given in Section 2, the iRAPM is proposed and its convergence analyzed in Section 3. Then, the method is specialized to the affine rank minimization problem in Section 4, where the new stopping criteria for the Krylov method used to compute the truncated SVD are presented, see Sections 4.2 and 4.3. In Section 5, numerical results on matrix completion problems are provided showing that the inexact approach is more efficient than the exact counterpart. 

%%%%%%%%%%%%%%%%%
% Preliminaries %
%%%%%%%%%%%%%%%%%
\section{Preliminaries}

In this section, we first introduce the notations and basic definitions used in the present work. Then, we recall the APM and RAPM algorithms and state their main convergence results in the case when one of the two sets is nonconvex.

\subsection{Notations and basic definitions} 
The symbol $\|\cdot\|$ denotes either the standard Euclidean norm on $\R^n$ or the Frobenius norm on $\R^{n_1\times n_2}$, depending on the context.  \REV{Analogously, $\langle \cdot,\cdot\rangle $ denotes either the Euclidean or the Frobenius inner product}.
Given two vectors $x,y \in \R^n$, we will use the notation $(x,y)$ as a shorthand for $(x^T,y^T)^T$. Given a non-empty, closed subset $A\subseteq \R^n$, the indicator function $\iota_A:\R^n\rightarrow \bR$ is given by
\begin{equation*}
    \iota_{A}(x)=\begin{cases}0, \quad &\text{if }x\in A\\
    +\infty, \quad &\text{otherwise},
    \end{cases}
\end{equation*}
where $\bR=\R\cup \{+\infty,\-\infty\}$ is the extended real numbers set.

Given a function $f: \R^n\rightarrow \bR$ and a set $A\subseteq \R^n$, $\underset{x\in A}{\operatorname{argmin}} \ f(x)$ denotes  the set of global minimum points of $f$ over the set $A$. The projection operator $\proj_A:\R^n\rightrightarrows \R^n$ onto the set $A$ is defined by
\begin{equation*}
    \proj_A(x) = \underset{y\in A}{\operatorname{argmin}} \ \frac{1}{2}\|y-x\|^2 = \underset{y\in \R^n}{\operatorname{argmin}} \ \frac{1}{2}\|y-x\|^2+\iota_A(y). 
\end{equation*}
We say that $A\subseteq \R^n$ is \emph{prox-regular} at a point $\bar{x}\in A$ if $\proj_A$ is single-valued around $\bar{x}$ \cite[Definition 2.6]{luke2013prox}. A set $A\subseteq \R^n$ is called \emph{semialgebraic} if it assumes the form \cite{Bolte2007}
\begin{equation}\label{eq:semialgebraic}
A = \bigcup_{i=1}^p\bigcap_{j=1}^q\{x\in\R^n: \  f_{i,j}(x) = 0, \ g_{i,j}(x) <0 \},
\end{equation}
where $f_{i,j},g_{i,j}:\R^n\rightarrow \R$ are polynomial functions for all $1\leq i \leq p$, $1\leq j\leq q$. Likewise, a function $f:\R^n\rightarrow\bR$ is called semialgebraic if its graph is a semialgebraic subset of $\R^n\times \R$.

Given a proper, lower semicontinuous function $f:\R^n\rightarrow \bR$, the proximal operator of $f$ is denoted by $\prox_f:\R^n\rightrightarrows \R^n$ and defined as
\begin{equation*}
    \prox_f(x)=\underset{y\in\R^n}{\operatorname{argmin} } \ \frac{1}{2}\|y-x\|^2+f(y).
\end{equation*}
Note that, for any non-empty, closed subset $A\subseteq \R^n$, we have $\proj_A(x) = \prox_{\iota_A}(x)$.

Let $f:\R^n\to\bR$ and $\bar{x}\in\dom(f)$. The \emph{Fr\'{e}chet subdifferential} and the \emph{limiting subdifferential} of $f$ at $\bar{x}$ are defined, respectively, as the sets \cite[Definition 8.3]{Rockafellar-Wets-1998}
\begin{align*}
\hat\partial f(\bar{x}) &= \left\{w\in\R^n : \liminf_{x\to \bar{x},x\neq \bar{x}} \frac {f(x)-f(\bar{x})-(x-\bar{x})^Tw}{\|x-\bar{x}\|}\geq 0\right\},\\
\partial f(\bar{x}) &= \{w\in\R^n : \exists \  x^k\to \bar{x}, \ f(x^k)\to f(\bar{x}), \ w^k \in \hat\partial f(x^k)\to w \ {\text{as}} \ k\to \infty\}.
\end{align*}
A point $\bar{x}\in\R^n$ is said to be \emph{stationary} for $f:\R^n\rightarrow \bR$ if $0\in\partial f(\bar{x})$. {We denote with $\|\partial f(x)\|_{-} = \inf_{v\in\partial f(x)}\|v\|$ the lazy slope of $f$ at $x$.}

Given a matrix $G \in \R^{n_1 \times n_2}$,  we denote its  Singular Value Decomposition (SVD) as
\[
G = U \Sigma V^T,
\]
where $U\in \R^{n_1 \times n_1}$, $V\in \R^{n_2 \times n_2}$ are orthonormal matrices and $\Sigma \in \R^{n_1 \times n_2}$ is the rectangular diagonal matrix \ste{whose diagonal entries $\sigma_1, \dots, \sigma_{\min(n_1,n_2)}$} are the singular values of $G$.
We will assume that the singular values are in decreasing order, that is, $\sigma_i \ge \sigma_{i+1}, \ i=1,\dots,\min(n_1,n_2)-1$. 
Given $1 \le r \le \min \{ n_1, n_2 \}$, we also consider the following partitions of matrices $U, V, \Sigma$: $U = [U_1 \ U_2]$, $V = [V_1 \ V_2]$, and $\Sigma = \diag(\Sigma_1, \Sigma_2)$, with $U_1 \in \R^{n_1 \times r}, U_2 \in \R^{n_1 \times (n_1-r)}, V_1 \in \R^{n_2 \times r}, V_2 \in \R^{n_2 \times (n_2-r)}, \Sigma_1 \in \R^{r \times r}, \Sigma_2 \in \R^{(n_1-r) \times (n_2-r)}$. With these partitions we can write
\begin{equation} \label{eq:svd_part}
    G = U_1 \Sigma_1 V_1^T + U_2 \Sigma_2 V_2^T = \Tsvd{G} + \tsvd{G}.
\end{equation}
The matrix $\Tsvd{G} = U_1 \Sigma_1 V_1^T$ is the {truncated SVD} to the first $r$ singular values, and in the following we will refer to it as $r$-truncated SVD. 

\subsection{Alternating projection methods for nonconvex feasibility problems}
Given two non empty, closed subsets $A\subseteq \R^n$, $B\subseteq \R^n$, we consider the feasibility problem \eqref{eq:problem} and the related minimization problem
\begin{equation}\label{eq:problem_L}
    \underset{(x,y)\in\R^n\times \R^n}{\operatorname{argmin}} \ L(x,y) \equiv \frac{1}{2}\|x-y\|^2+\iota_A(x)+\iota_B(y).
\end{equation}
The classical {\it alternating projection method} (APM) for solving \eqref{eq:problem} was first presented by J. von Neumann \cite{Neumann-etal-1950}. Such a method can be seen as the standard alternating minimization (or Gauss-Seidel) method applied to problem \eqref{eq:problem_L}, where the function is cyclically minimized with respect to each block of variables, while keeping the other one fixed. Then, given $y^0\in B$, it generates two sequences $\{x^k\}_{k\in\N}\subseteq A$ and $\{y^k\}_{k\in\N}\subseteq B$ defined as follows:
\begin{align} 
    x^{k+1}&\in \underset{x\in\R^n}{\operatorname{argmin}} \ L(x,y^k)=\underset{x\in\R^n}{\operatorname{argmin}} \ \frac{1}{2}\|x-y^k\|^2+\iota_A(x)=\proj_A(y^k)\label{APM1}\\
    y^{k+1}&\in \underset{y\in\R^n}{\operatorname{argmin}} \ L(x^{k+1},y)=\underset{y\in\R^n}{\operatorname{argmin}} \ \frac{1}{2}\|x^{k+1}-y\|^2+\iota_B(y)=\proj_B(x^{k+1}) \label{APM2}
\end{align}

Although global convergence of method \eqref{APM1}-\eqref{APM2} to a solution of problem \eqref{eq:problem} holds whenever $A$ and $B$ are convex, the same property can not be guaranteed in the case where either one of or both sets are nonconvex \cite{bauschke2013, Noll-etal-2016}. On the other hand, the local convergence properties of APM in the nonconvex setting have been investigated {under additional assumptions on the sets of interest. In \cite{lewis2008alternating}, the authors prove the local linear convergence rate of APM whenever two smooth manifolds intersect transversally. This result is further generalized in \cite{lewis2009local} for any two sets intersecting transversally, provided that at least one of them is superregular at an intersecting point. In \cite{drusvyatskiyPhDthesis2013} (see also the subsequent work \cite{drusvyatskiy2015transversality}), the authors remove the superregularity assumption and prove the local linear convergence of APM by only requiring transversality (or a slightly weaker version of it). We note that, in practice, checking whether transversality holds requires the knowledge of a point in the intersection, even for semialgebraic sets \cite[p. 1648]{drusvyatskiy2015transversality}. In \cite{Noll-etal-2016}, the authors prove the local convergence of APM at a sublinear rate, by assuming that the two sets intersect separably and  at least one of them is $\sigma-$\emph{H\"older regular}, with $\sigma \in[0,1)$. Although the convergence rate in \cite{Noll-etal-2016} is weaker than those obtained in \cite{drusvyatskiyPhDthesis2013,drusvyatskiy2015transversality,lewis2009local,lewis2008alternating}, the separable intersection condition automatically holds when the sets are semialgebraic \cite[Theorem 3]{Noll-etal-2016}, which is the case of interest for affine rank minimization problems (see Section~\ref{sec:application} for an in-depth analysis).}

The definitions of separable intersection and H\"older regularity, as well as the local convergence result derived in \cite{Noll-etal-2016}, are reported below.

\begin{definition}\label{def:separable_intersection}\cite[Definition 1]{Noll-etal-2016}
    Two sets $A$ and $B$ intersect separably at $\bar{x}\in A\cap B$ with exponent $\omega\in[0,2)$ and constant $\gamma>0$ if there exists a neighbourhood $\mathcal{U}$ of $\bar{x}$ such that, for any points $a\in A$, $b\in\mathcal{P}_B(a)$, $a^+\in\mathcal{P}_A(b)$, $b^+\in\mathcal{P}_B(a^+)$ belonging to $\mathcal{U}$, the following condition holds
    \begin{equation*}
        \langle b-a^+,b^+-a^+\rangle\leq (1-\gamma\|b^+-a^+\|^\omega)\|b-a^+\|\|b^+-a^+\|.
    \end{equation*}
\end{definition}

\begin{definition}\label{def:holder}\cite[Definition 2]{Noll-etal-2016}
    Let $\sigma\in [0,1)$. The set $B$ is $\sigma-$\emph{H\"older regular} with respect to $A$ at $b^*\in A\cap B$ if there exists a neighborhood $\mathcal{U}$ of $b^*$ and a constant $c>0$ such that for every $a^+\in A\cap \mathcal{U}$ and $b\in\proj_B(a^+)\cap \mathcal{U}$, setting $r = \|a^+-b^+\|$, one has
    \begin{equation*}
        \mathcal{B}(a^+,(1+c)r)\cap \{b\in \proj_A(a^+)^{-1}: \ \langle a^+-b^+,b-b^+\rangle > \sqrt{c}r^{\sigma+1}\|b-b^+\|\}\cap B =\emptyset,
    \end{equation*}
    where $\mathcal{B}(a^+,(1+c)r)$ denotes the ball of center $a^+$ and radius $(1+c)r$.
\end{definition}

\begin{theorem}\cite[Theorem 1, \REV{Corollary 4}]{Noll-etal-2016}\label{thm:conv_APM} (Local convergence)
    Suppose $B$ intersects $A$ separably at $\bar{x}\in A \cap B$ with exponent $\omega\in\REV{(}0,2)$ and constant $\gamma>0$. Assume also that $B$ is $\omega/2-$H\"older regular at $\bar{x}$ with respect to $A$ and constant $c<\frac{\gamma}{2}$. Then there exists a neighborhood $\mathcal{V}$ of $\bar{x}$ such that any sequences $\{x^k\}_{k\in\mathbb{N}}$ and $\{y^k\}_{k\in\mathbb{N}}$ generated by \eqref{APM1}-\eqref{APM2} with $y^0\in \mathcal{V}$ converges to a point $\bar{b}\in A\cap B$. \REV{Furthermore, $\|x^k-\bar{b}\|=\mathcal{O}(k^{-\frac{2-\omega}{2\omega}})$ and $\|y^k-\bar{b}\|=\mathcal{O}(k^{-\frac{2-\omega}{2\omega}})$.}
\end{theorem}

In \cite{attouch2010proximal,attouch2007}, the authors investigate on the convergence of a regularized modification of APM, obtained by proximal modification of the two alternating steps. Given two strictly positive sequences $\{\lambda_k\}_{k\in\N}$ and $\{\mu_k\}_{k\in\N}$, the regularized APM generates two sequences $\{x^k\}_{k\in\N}\subseteq A$ and $\{y^k\}_{k\in\N}\subseteq B$  such that:
\begin{align*}
     x^{k+1}&\in \prox_{\lambda_k L(\cdot,y^k)}(x^k)=\underset{x\in \R^n}{\operatorname{argmin}} \ \frac{1}{2}\|x-y^k\|^2+\frac{1}{2\lambda_k}\|x-x^k\|^2+\iota_A(x), \\
    y^{k+1}&\in \prox_{\mu_k L(x^{k+1},\cdot)}(y^k)=\underset{y\in \R^n}{\operatorname{argmin}} \ \frac{1}{2}\|y-x^{k+1}\|^2+\frac{1}{2\mu_k}\|y-y^k\|^2+\iota_B(y).
\end{align*}

By means of some easy algebraic manipulations, we can interpret the points $x^{k+1}$ and $y^{k+1}$ as projections onto $A$ and $B$, respectively, of weighted means of the current iterates, i.e.,
\begin{equation*}
     x^{k+1}\in\proj_A\left(\frac{1}{\lambda_k+1}x^k+\frac{\lambda_k}{\lambda_k+1}y^k\right), \quad y^{k+1}\in \proj_B\left(\frac{1}{\mu_k+1}y^k+\frac{\mu_k}{\mu_k+1}x^{k+1}\right).
\end{equation*}
The regularized APM (RAPM)  is sketched  in Algorithm~\ref{algo:APM_regu}.

\begin{algorithm}[h!]\caption{Regularized Alternating Projection Method (RAPM)\cite {attouch2010proximal,attouch2007}}\label{algo:APM_regu}
    Choose $(x^0,y^0)\in A\times B$, 
    $0<\lambda_{-}<\lambda_+$, $0<\mu_{-}<\mu_+$, $\{\lambda_k\}_{k\in\mathbb{N}}\subseteq [\lambda_{-},\lambda_+]$, $\{\mu_k\}_{k\in\mathbb{N}}\subseteq [\mu_{-},\mu_{+}]$.
    
    {\textsc{For} $k=0,1,\ldots$}
    \begin{itemize}[leftmargin=2cm]
        \item[{\sc Step 1}] Compute $x^{k+1} \in \proj_A\left(\frac{1}{\lambda_k+1}x^k+\frac{\lambda_k}{\lambda_k+1}y^k\right)$.
        \item[{\sc Step 2}] Compute $y^{k+1} \in \proj_B\left(\frac{1}{\mu_k+1}y^k+\frac{\mu_k}{\mu_k+1}x^{k+1}\right)$.
    \end{itemize}
\end{algorithm}
In \cite{attouch2010proximal}, the authors carry out a convergence analysis for RAPM in the nonconvex setting, by just assuming that the function $L$ satisfies the so-called Kurdyka-\L{}ojasiewicz inequality, which is reported below for the reader's convenience.

\begin{definition}\label{def:KL_original}
    Let $F:\R^n \longrightarrow \bR$  be a proper, lower semicontinuous function. The function $F$ is said to have the Kurdyka-\L{}ojasiewicz (KL) property at $\bar{x} \in \dom(\partial F)$ if there exist $\upsilon \in (0,+\infty]$, a neighborhood $\mathcal{U}$ of  $\bar{x}$ and a continuous concave function $\phi:[0,\upsilon)\longrightarrow [0,+\infty)$ with $\phi(0) = 0$, $\phi\in C^1(0,\upsilon)$, $\phi'(s) > 0$ for all $s \in (0,\upsilon)$, such that 
    \begin{equation*}
        \phi'(F(z)-F(\bar{x})) \|\partial F(z)\|_{-} \geq 1,
    \end{equation*}
    for all $z\in \mathcal{U}\cap \{z\in\R^n: \ F(\bar{x})<F(z)<F(\bar{x})+ \upsilon \}$. If $F$ satisfies the KL property at each point of $\operatorname{dom}(\partial F)$, then $F$ is called a KL function.
\end{definition}

The convergence results proved in \cite{attouch2010proximal} are reported in Theorem~\ref{thm:conv_APM_regu} for the reader's convenience. Unlike APM, the RAPM algorithm shows both local and global convergence, and its convergence rate depends on the geometrical shape of the desingularizing function $\phi$ in Definition~\ref{def:KL_original}.

\begin{theorem}\cite[Corollary 3.1, Theorem 3.4]{attouch2010proximal}\label{thm:conv_APM_regu}
    Let $\{x^k\}_{k\in\mathbb{N}}$ and $\{y^k\}_{k\in\mathbb{N}}$ be the sequences generated by Algorithm RAPM.
    \begin{itemize}
        \item[(i)] (Global convergence) If the function $L$ defined in \eqref{eq:problem_L} is a KL function, then either $\|(x^k,y^k)\|\rightarrow \infty$, or the sequence $\{(x^k,y^k)\}_{k\in\mathbb{N}}$ converges to a stationary point of $L$.
        \item[(ii)] (Local convergence) Suppose that the function $L$ defined in \eqref{eq:problem_L} has the KL property at a point $(\bar{x},\bar{y})$ and that $\|\bar{x}-\bar{y}\|=\min\{\|x-y\|: \ x\in A, \ y\in B\}$. If $(x^0,y^0)$ is sufficiently close to $(\bar{x},\bar{y})$, then the whole sequence $\{(x^k,y^k)\}_{k\in\mathbb{N}}$ converges to a point $(x_{\infty},y_{\infty})$ such that $\|x_{\infty}-y_{\infty}\|=\min\{\|x-y\|: \ x\in A, \ y\in B\}$.
        \item[(iii)] (Rate of convergence) Suppose that $\{(x^k,y^k)\}_{k\in\N}$ converges to a point $(x^*,y^*)\in A\times B$ and that the function $L$ defined in \eqref{eq:problem_L} has the KL property at $(x^*,y^*)$ with $\phi(t)=ct^{1-\theta}$, being $c>0$ and $\theta\in[0,1)$. Then, if $\theta=0$, the sequence $\{(x^k,y^k)\}_{k\in\N}$ terminates in a finite number of iterations; otherwise, we have
        \begin{equation*}
            \|(x^k,y^k)-(x^*,y^*)\| = \begin{cases}
               \mathcal{O}(\tau^k), \ \text{for some }\tau\in[0,1), \quad &\text{if }\theta\in(0,\frac{1}{2}],\\
                \mathcal{O}\left(k^{-(1-\theta)/(2\theta-1)}\right), \quad &\text{if }\theta\in(\frac{1}{2},1).
            \end{cases}
        \end{equation*}     
    \end{itemize}
\end{theorem}

%%%%%%%%%%%%%%%%%%
% Inexact Method %
%%%%%%%%%%%%%%%%%%
\section{An inexact regularized alternating projection method}

In this section, we propose an inexact version of Algorithm RAPM, which is designed for solving the feasibility problem \eqref{eq:problem} whenever the projection onto $B$ is costly or does not admit a closed-form implementation. After presenting the method in full detail, we discuss its practical implementation and analyze its convergence in a nonconvex scenario. Regarding the latter aspect, we prove the convergence of the iterates to a stationary limit point of the function $L$ and investigate on  the  convergence rates of the distance to the limit, by assuming that a specific merit function related to $L$ satisfies the KL property on its domain.

\subsection{The proposed method}
We call our proposed method iRAPM - inexact Regularized Alternating Projection Method. iRAPM alternates between an exact (regularized) projection onto $A$ and an inexact (regularized) projection onto $B$, generating two sequences $\{x^k\}_{k\in\N}\subseteq A$ and $\{y^k\}_{k\in\N}\subseteq B$. At each iteration $k\in\N$, given $x^k,y^k$ and two positive parameters $\lambda_k,\mu_k$, the next iterate $x^{k+1}\in A$ is computed as an (exact) projected point of the convex combination $(x^k+\lambda_k y^k)/(\lambda_k+1)$ onto $A$, i.e.,
\begin{equation*}
    x^{k+1} \in \proj_A\left(\frac{1}{\lambda_k+1}x^k+\frac{\lambda_k}{\lambda_k+1}y^k\right),
\end{equation*}
whereas the next iterate $y^{k+1}\in B$ represents an inexact projection of the point $(y^k+\mu_k x^{k+1})/(\mu_k+1)$ onto $B$. In order to introduce the inexactness criteria, we let $Q^{k+1}$ be the function defined as
\begin{equation}
Q^{k+1}(y):= \frac{1}{2\mu_k}\|y-y^k\|^2+L(x^{k+1},y)-L(x^{k+1},y^k).\label{eq:Q}
\end{equation}
{Let us denote with $\yck$ the point $(y^k+\mu_k x^{k+1})/(\mu_k+1)$. Then, we can conveniently write $Q^{k+1}$ in \eqref{eq:Q} as}
\begin{align}\label{eq:Qbis}
    Q^{k+1}(y)&= 
    \frac{1}{2\mu_k}\|y-y^k\|^2+\frac{1}{2}\|x^{k+1}-y\|^2-\frac{1}{2}\|x^{k+1}-y^k\|^2+\iota_B(y) \nonumber  \\
    &=\frac{(1+\mu_k)}{2\mu_k}\|y-y^k\|^2-\frac{1}{2}\|y-y^k\|^2+\frac{1}{2}\|x^{k+1}-y\|^2-\frac{1}{2}\|x^{k+1}-y^k\|^2+\iota_B(y)\nonumber\\
    &=\frac{(1+\mu_k)}{2\mu_k}\|y-y^k\|^2+\langle x^{k+1}-y^k,y^k-y\rangle+\iota_{B}(y)\nonumber\\
    &=\frac{(1+\mu_k)}{2\mu_k}\|y-y^k\|^2+\frac{1+\mu_k}{\mu_k}\left\langle \frac{\mu_k}{1+\mu_k}(y^k-x^{k+1}),y-y^k\right\rangle+\iota_{B}(y)\nonumber\\
    &=\frac{(1+\mu_k)}{2\mu_k}\left\|y-y^k+\frac{\mu_k}{1+\mu_k}\left(y^k-x^{k+1}\right)\right\|^2-\frac{(1+\mu_k)}{2\mu_k}\left\|\frac{\mu_k}{1+\mu_k} (y^k - x^{k+1})\right\|^2+\iota_B(y)\nonumber\\
    &=\frac{(1+\mu_k)}{2\mu_k}\left\|y-\yck\right\|^2-\frac{(1+\mu_k)}{2\mu_k}\left\|y^k-\yck\right\|^2+\iota_B(y), \quad \forall \ y\in \R^n,
\end{align}
{where the first equality follows from the form of $L$ given in \eqref{eq:problem_L}, the third and fifth are both due to the application of the basic relation $\|a+b\|^2=\|a\|^2+\|b\|^2+2\langle a,b\rangle$, and the sixth follows by using the definition of $y_{reg}^k$.}
Hence the set of minimum points of $Q^{k+1}$ is the projection set of $\yck$ onto $B$.
Letting $\hat{y}^{k+1}\in\mathcal{P}_B(\yck)$, we observe that the value $Q^{k+1}(\hat{y}^{k+1})$ assumes negative or null sign, since
\begin{equation}\label{eq:sign-Q}
    Q^{k+1}(\hat{y}^{k+1})=\min_{y\in\R^n}Q^{k+1}(y)\leq Q^{k+1}(y^k)=0.
\end{equation}

Then, given a prefixed tolerance parameter $\zeta\in(0,1]$, iRAPM computes an approximate projection $y^{k+1}\in B$ such that the following conditions are satisfied:
\begin{align}
    Q^{k+1}(y^{k+1})&\leq \zeta 
    Q^{k+1}(\hat{y}^{k+1})\label{eq:inexact}\\
    \|y^{k+1}-\hat{y}^{k+1}\|&\leq \sqrt{-\left(\frac{1-\zeta}{\zeta}\right)Q^{k+1}(y^{k+1})},\label{eq:bound}
\end{align}

We report iRAPM in its entirety in Algorithm~\ref{algo:APM_regu_inexact}. 
\begin{algorithm}[h!]\caption{inexact Regularized Alternating Projection Method (iRAPM)}\label{algo:APM_regu_inexact}
    Choose $(x^0,y^0) \in A\times B$, $0<\lambda_{-}<\lambda_+$, $0<\mu_{-}<\mu_+$, $\{\lambda_k\}_{k\in\mathbb{N}}\subseteq [\lambda_{-},\lambda_+]$, $\{\mu_k\}_{k\in\mathbb{N}}\subseteq [\mu_{-},\mu_{+}]$, $\zeta\in(0,1]$.\\
    
    {\textsc{For} $k=0,1,\ldots$}
    \begin{itemize}[leftmargin=2cm]
        \item[{\sc Step 1}] Compute $x^{k+1} \in \proj_A\left(\frac{1}{\lambda_k+1}x^k+\frac{\lambda_k}{\lambda_k+1}y^k\right)$.
        \item[{\sc Step 2}] Compute $y^{k+1}\in B$ such that \eqref{eq:inexact}-\eqref{eq:bound} hold.
    \end{itemize}
\end{algorithm}

\REV{Some comments on the inexactness conditions \eqref{eq:inexact}-\eqref{eq:bound} employed in Algorithm~\ref{algo:APM_regu_inexact} are in order}. 
\begin{itemize}
    \item Condition \eqref{eq:inexact} is well defined due to \eqref{eq:sign-Q}. Furthermore,  $Q^{k+1}(y^{k+1})\leq 0$ and if $\zeta=1$, the point $y^{k+1}$ belongs to the (exact) projection set of $\yck$ onto $B$; otherwise, if $\zeta\in (0,1)$, $y^{k+1}$ represents an appropriate approximation that is ``sufficiently close'' to a point of the projection set. The parameter $\zeta$ controls the approximation level of the iterates $y^{k+1}$; the closer $\zeta$ is to $1$, the closer $y^{k+1}$ is to the projection set, whereas the closer $\zeta$ is to $0$, the coarser $y^{k+1}$ is as an inexact projection. Condition \eqref{eq:inexact} has been frequently used for computing inexact projections \cite{Birgin2003,Ferreira-2022} and inexact proximal points \cite{Bonettini-Prato-Rebegoldi-2021} in the convex setting;
    \item \REV{A bound similar to the one in \eqref{eq:bound} is automatically satisfied if condition \eqref{eq:inexact} holds and $B$ is convex; indeed, in this case, we can combine the $\frac{1+\mu_k}{\mu_k}-$strong convexity of $Q^{k+1}$ and \eqref{eq:inexact} to obtain
    \begin{equation}\label{eq:bound_bis}
    \|y^{k+1}-\hat{y}^{k+1}\|\leq \sqrt{\frac{2\mu_k}{1+\mu_k} (Q^{k+1}(y^{k+1})-Q^{k+1}(\hat{y}^{k+1}))}\leq \sqrt{-\frac{2\mu_k(1-\zeta)}{(1+\mu_k)\zeta} Q^{k+1}(y^{k+1})}.
    \end{equation}
    However, if $B$ is not convex, then $Q^{k+1}$ is not strongly convex and condition \eqref{eq:bound} needs to be numerically imposed in order to guarantee the convergence of Algorithm~\ref{algo:APM_regu_inexact}.}
    {%
    \item In Figure~\ref{fig:counterexample}, we present an example showing that the two inexactness conditions \eqref{eq:inexact}-\eqref{eq:bound} may be independent one of each other whenever $B$ is nonconvex.
    Let $A = [0,2]\subseteq \R$ and $B = [0,1]\cup [2,3]\subseteq \R$. Note that $B$ is the union of two disjoint intervals, hence it is nonconvex. Given $(x^0,y^0)=(0,3)\in A\times B$, $\lambda_0=1$, $\mu_0 = 10$ and $\zeta=\frac{1}{2}$, it is easy to show that $x^0_{reg} = \frac{3}{2}$, $y^0_{reg} = \frac{18}{11} \approx 1.6364$, and $\hat{y}^1 = 2$. 
    In the figure, the red solid parabola represents the function $Q^1(y)$ for $y \in B$, whereas the red dotted parabola is its extension for $y \notin B$ obtained by neglecting $\iota_B(y)$,
    the cyan solid line is the right-hand term of condition \eqref{eq:inexact}, while the green dashed and magenta dash-dotted lines represent the left and right-hand sides of condition \eqref{eq:bound}, respectively.
    In this case, {the minimum of $Q^1$ is attained at a point $\hat{y}^1$ that is different from $y^0_{reg}$.} %the minimum of the red curve is attained at $y^0_{reg}$ which is different from $\hat{y}^1$, that is, the minimum of $Q^1(y)$.
    The subintervals of $B$ satisfying conditions \eqref{eq:inexact}-\eqref{eq:bound} are represented as cyan and green segments on the horizontal axis, respectively. We see that {there are no points of $B$ to the left of $y_{reg}^0$ that satisfy condition \eqref{eq:bound}.} 
    This means that there exists a point $y \in B$ such that it satisfies condition \eqref{eq:inexact} but not condition \eqref{eq:bound}, showing that the second condition does not automatically follow from the first one in the nonconvex case. Conversely, the cyan interval on the right interval is strictly included in the green one, which implies that the first condition does not immediately follow from the second one. {Furthermore, in these settings, it is possible to show that condition \eqref{eq:inexact} is also independent of the bound \eqref{eq:bound_bis}, which would hold if the set $B$ were convex, as discussed in the previous bullet.}
    }
    \begin{figure}
        \centering
        \includegraphics[width=0.6\linewidth]{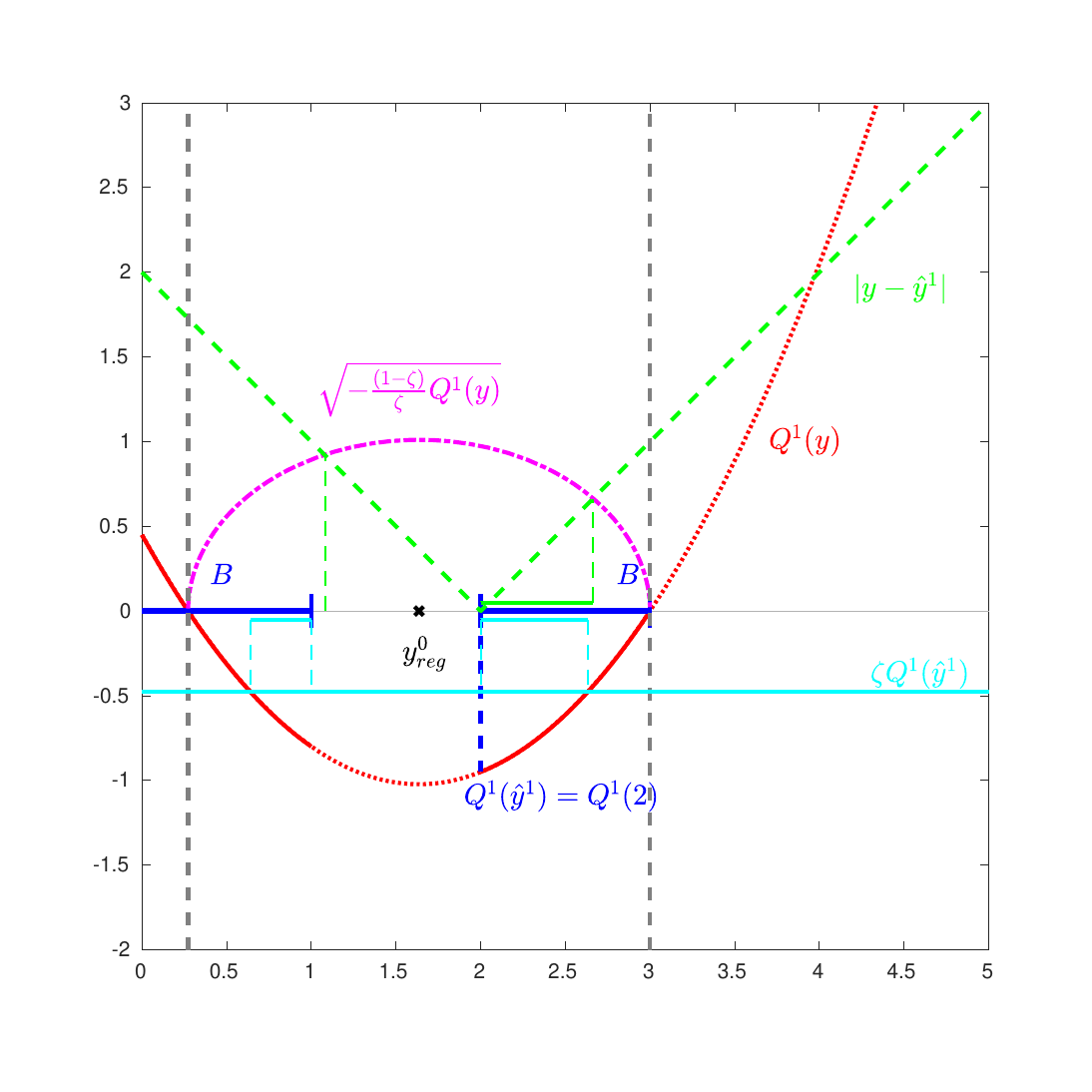}
        \caption{One-dimensional example showing that the two inexactness conditions \eqref{eq:inexact}-\eqref{eq:bound} are independent one of each other in the nonconvex case}
        \label{fig:counterexample}
    \end{figure}
\end{itemize}

\subsection{Practical computation of the inexact projections}\label{sec:practical}
We now discuss how we can practically implement the regularized inexact projection appearing at Step 2 of Algorithm~\ref{algo:APM_regu_inexact}. At first glance, it looks like enforcing the inexactness conditions \eqref{eq:inexact}-\eqref{eq:bound} requires the knowledge of the projection set of $\yck$ onto $B$, which we assumed we do not have access to. Nonetheless, we show  how to get rid of such an information and still compute a point $y^{k+1}\in B$ satisfying \eqref{eq:inexact}-\eqref{eq:bound}.

\subsubsection{Enforcing condition \eqref{eq:inexact}}
The following discussion is borrowed from \cite[Section 2.1.1]{Ferreira-2022}, which in turn revises a procedure that has been formerly proposed in \cite{Birgin2003}.
{By employing the form of $Q^{k+1}$ given in \eqref{eq:Qbis}}, we can easily rewrite the inexactness condition \eqref{eq:inexact} as
\begin{equation}\label{eq:inexact_ferreira}
    \|y^{k+1}-\yck\|^2\leq \zeta \|\hat{y}^{k+1}-\yck\|^2+(1-\zeta)\|y^k-\yck\|^2.
\end{equation}
Condition \eqref{eq:inexact_ferreira} is identical to the concept of {\it feasible inexact projection of $\yck$ onto $B$ relative to the point $y^k\in B$ and forcing parameter $\zeta\in(0,1]$ with respect to the norm $\|\cdot\|$}, which is proposed in \cite[Definition 2]{Ferreira-2022} for convex sets. According to \cite[Section 2.1.1]{Ferreira-2022},  it is possible to define a practical procedure for computing a point $y^{k+1}$ satisfying \eqref{eq:inexact_ferreira} if there exist two computable, iteration-dependent sequences $\{c^\ell\}_{\ell \in \mathbb{N}} \subseteq \R$, $\{w^\ell\}_{\ell \in \mathbb{N}} \subseteq B$ complying with the following conditions
\begin{subequations}
    \begin{align}
        c^\ell &\le \|\hat{y}^{k+1}-\yck\|^2, \quad \forall \ \ell \in \mathbb{N} \label{eq:cond_1} \\
        \lim_{\ell \rightarrow +\infty} c^\ell & = \|\hat{y}^{k+1}-\yck\|^2 \label{eq:cond_2} \\
        \lim_{\ell \rightarrow +\infty} w^\ell & = \hat{y}^{k+1}. \label{eq:cond_3}
    \end{align}
\end{subequations}
More precisely, the following result holds. The  proof, in the context of inexact projection onto convex sets is given in \cite[Section 2.1.1]{Ferreira-2022}; we report the proof in the appendix for sake to completeness and to enlighten that the result is valid also if $B$ is not convex. 
\begin{lemma}\label{lem:stopping}
    Let $\{c^\ell\}_{\ell \in \mathbb{N}} \subseteq \R$, $\{w^\ell\}_{\ell \in \mathbb{N}} \subseteq B$ be two sequences complying with \eqref{eq:cond_1}-\eqref{eq:cond_2}-\eqref{eq:cond_3}. Then, there exists $\hat{\ell}\in \N$ such that
    \begin{equation}\label{eq:stopping}
        \|w^{{\ell}}-\yck\|^2\leq \zeta c^{{\ell}}+(1-\zeta)\|y^k-\yck\|^2,\quad \forall \ \ell\ge \hat{\ell},
    \end{equation}
    which implies that condition \REV{\eqref{eq:inexact}} holds with $y^{k+1}=w^{{\ell}}$ for $\ell\ge \hat{\ell}$.  
\end{lemma}

\begin{proof}
    See Appendix~\ref{appendix:A1}.
\end{proof}

The above procedure has been previously applied only to convex sets with a specific structure \cite{Birgin2003,Ferreira-2022}. In Section~\ref{sec:application}, we show how we can find the sequences $\{c^{\ell}\}_{\ell\in\N}$, $\{w^{\ell}\}_{\ell\in\N}$ for computing inexact projections over the rank level set $\mathcal{C}_r = \{Z\in\R^{n_1\times n_2}: \ \operatorname{rank}(Z)\leq r\}$, which is nonconvex.

\subsubsection{Enforcing condition \eqref{eq:bound}}
We can implement condition \eqref{eq:bound} whenever we have at our disposal a vanishing sequence $\{a^{\ell}\}_{\ell\in\N}\subseteq \R$ representing an upper bound for the quantity $\|w^{\ell}-\hat{y}^{k+1}\|$, i.e.
\begin{equation}\label{eq:cond_4}
    \|w^{\ell}-\hat{y}^{k+1}\|\leq a^{\ell}, \quad \lim_{\ell\rightarrow \infty}a^{\ell} = 0.  
\end{equation}
In this case, the following result holds.

\begin{lemma}\label{lem:stopping_2}
    Let $\{w^{\ell}\}_{\ell\in\N}\subseteq B$ and $\{a^{\ell}\}_{\ell\in\N}\subseteq \R$ be sequences complying with \eqref{eq:cond_3} and \eqref{eq:cond_4}, respectively.
    Then there exists $\tilde{\ell}\in\N$ such that
    \begin{equation}\label{eq:stopping_2}
        Q^{k+1}(w^{{\ell}})< 0, \quad a^{{\ell}}\leq \sqrt{-\left(\frac{1-\zeta}{\zeta}\right)Q^{k+1}(w^{{\ell}})} \quad \forall \ell\ge \tilde \ell,
    \end{equation}
    which implies that condition \eqref{eq:bound} holds with $y^{k+1} = w^{\ell}$ for $\ell\ge \tilde \ell$.
\end{lemma}

\begin{proof}
If $y^k \in  \mathcal{P}_{B}(\yck)$,  one can always set $w^{\ell} = y^k$ and   $a^{\ell}=0$ for any $\ell$, so that conditions \eqref{eq:cond_3}, \eqref{eq:cond_4} and \eqref{eq:stopping_2} are trivially satisfied for all $\ell\in\N$ with $\hat{y}^{k+1}=y^k $. Then, without loss of generality, we assume that $y^k\notin \mathcal{P}_{B}(\yck)$.
Since $y^k \notin \mathcal{P}_B(\yck)$, we have $Q^{k+1}(\hat{y}^{k+1})<0$ by definition of $Q^{k+1}$. Then, since condition \eqref{eq:cond_3} holds and $Q^{k+1}$ is a continuous function on $B$, we have 
\begin{equation*}
    \lim_{\ell\rightarrow \infty}Q^{k+1}(w^{\ell})=Q^{k+1}(\hat{y}^{k+1})<0.
\end{equation*}
Then, for any $0<\epsilon<-Q^{k+1}(\hat{y}^{k+1})$, there exists $\ell_\epsilon>0$ such that
\begin{equation*}
    0<-Q^{k+1}(\hat{y}^{k+1})-\epsilon\leq -Q^{k+1}(w^\ell), \quad \forall \ \ell\geq \ell_{\epsilon}. 
\end{equation*}
Likewise, due to the right part of \eqref{eq:cond_4}, there exists $\tilde{\ell}\ge \ell_\epsilon$ sufficiently large such that
\begin{equation*}
    a^{{\ell}}\leq \sqrt{\left(\frac{1-\zeta}{\zeta}\right)(-Q^{k+1}(\hat{y}^{k+1})-\epsilon)}\leq\sqrt{-\left(\frac{1-\zeta}{\zeta}\right)Q^{k+1}(w^{\tilde{\ell}})}, \quad \forall \ \ell>\tilde \ell
\end{equation*}
so that \eqref{eq:stopping_2} is met and the inexactness condition \eqref{eq:bound} holds with $y^{k+1} = w^{{\ell}}$ for $\ell\ge \tilde \ell$.
\end{proof}
Once again, it is possible to compute a sequence $\{a^{\ell}\}_{\ell\in\N}\subseteq \R$ complying with \eqref{eq:cond_4} if $B$ is the rank level set; see Section~\ref{sec:application} for the details.

\subsubsection{Computation of points satisfying \eqref{eq:inexact}-\eqref{eq:bound}}
Lemma~\ref{lem:stopping} and Lemma~\ref{lem:stopping_2} provide us with an implementable procedure for computing a point $y^{k+1}$ satisfying simultaneously the inexactness conditions \eqref{eq:inexact}-\eqref{eq:bound} employed in Algorithm~\ref{algo:APM_regu_inexact}. Indeed, given three sequences $\{c^{\ell}\}_{\ell\in\N}$, $\{w^{\ell}\}_{\ell\in\N}$, $\{a^{\ell}\}_{\ell\in\N}$ complying with \eqref{eq:cond_1}-\eqref{eq:cond_2}-\eqref{eq:cond_3}-\eqref{eq:cond_4} at any given iteration $k\in\N$, we compute such sequences until the occurrence of the first iteration $\bar{\ell}$ satisfying
\begin{equation}\label{eq:stopping_final}
    \begin{cases}
        \|w^{\bar{\ell}}-{\yck}\|^2\leq \zeta c^{\bar{\ell}}+(1-\zeta)\|y^k-{\yck}\|^2\\
        a^{\bar{\ell}}\leq \sqrt{-\left(\frac{1-\zeta}{\zeta}\right)Q^{k+1}(w^{\bar{\ell}})}.
    \end{cases}
\end{equation}
Such a procedure is well-defined thanks to Lemma~\ref{lem:stopping} and Lemma~\ref{lem:stopping_2}. Then, the inexact projection complying with \eqref{eq:inexact}-\eqref{eq:bound} is given by $y^{k+1}=w^{\bar{\ell}}$.

\subsection{Convergence analysis}
In the following, we prove the convergence of the iterates sequence generated by Algorithm~\ref{algo:APM_regu_inexact} towards a stationary point of $L$, and derive some convergence rates on the gap between the iterates and their limit point. Our analysis is inspired by \cite{attouch2010proximal} and will be carried out by assuming that a certain merit function related to $L$ is a KL function. The presence of an inexact projection complying to conditions \eqref{eq:inexact}-\eqref{eq:bound} leads to some major changes in the convergence proof of Algorithm~\ref{algo:APM_regu_inexact} in comparison to the one available in \cite{attouch2010proximal}.

For our purposes, we introduce the merit function $F:\R^n\times \R^n\times \R\rightarrow \R$ defined as
\begin{equation}\label{eq:merit}
    F(x,y,\rho) = L(x,y) +\frac{1}{2}\rho^2.
\end{equation}
Our analysis relies on the following set of properties, which hold for the sequences generated by iRAPM and involve both the objective function $L$ and the related merit function $F$.

\begin{lemma}\label{lem:properties}
    Let $\{(x^k,y^k)\}_{k\in\N}$ be the sequence generated by Algorithm~\ref{algo:APM_regu_inexact} and let $\{\hat{y}^{k}\}_{k\in\N}$ be any sequence with $\hat{y}^{k+1}\in\mathcal{P}_B(\yck)$ for all $k\in\N$. Furthermore, define the sequences $\{d_k\}_{k\in\N}$, $\{\gamma_k\}_{k\in\N}$ and $\{\rho_k\}_{k\in\N}$ such that, for all $k\in\N$,
    \begin{align}
        d_k &= \sqrt{\|(x^{k+1}-x^k,y^{k+1}-y^k)\|^2-Q^{k+1}(y^{k+1})}, \label{eq:dk}\\
        \gamma_k&=-\left(\frac{1-\zeta}{\zeta}\right)Q^{k+1}(\hat{y}^{k+1}),\label{eq:rk}\\
        \rho_{k} &= \sqrt{\frac{1}{\mu_{k-1}}\|\hat{y}^{k}-y^{k-1}\|^2-2\left(\frac{1-\zeta}{\zeta}\right)Q^{k}(y^{k})}.\label{eq:rhok}
    \end{align}
    The following statements hold true.
    \begin{itemize}
        \item[(i)] We have
        \begin{equation}\label{eq:ineq_dk}
        \|(x^{k+1}-x^k,y^{k+1}-y^k)\| \leq d_k, \quad \sqrt{-Q^{k+1}(y^{k+1})} \leq d_k, \quad \forall \ k\in\N.
        \end{equation}
        \item[(ii)] There exists $a>0$ such that
        \begin{equation*}
            L(x^{k+1},y^{k+1})+ad_k^2\leq L(x^k,y^k), \quad \forall \ k\in \N.
            \end{equation*}    
            \item[(iii)] We have 
            \begin{equation}\label{eq:limit3}
                L(x^{k+1},y^{k+1})\leq F(x^{k+1},\hat{y}^{k+1},{\rho_{k+1}})\leq L(x^k,y^k)+\gamma_k, \quad \forall \ k\in \N. 
            \end{equation}
            Furthermore $\lim\limits_{k\rightarrow \infty} {\rho_k}=\lim\limits_{k\rightarrow \infty}\gamma_k=0$, and
            \begin{equation}\label{eq:limit2}
            \lim_{k\rightarrow \infty}\|\hat{y}^{k+1}-y^{k+1}\|=0.
            \end{equation}
            \item[(iv)] There exists $b>0$ such that 
            \begin{equation}\label{bound_sub}
                \|\partial F(x^{k+1},\hat{y}^{k+1},{\rho_{k+1}})\|_{-}\leq b d_k, \quad \forall \ k\in \N.
            \end{equation}
    \end{itemize} 
\end{lemma}

\begin{proof}
See Appendix~\ref{appendix:A2}.
\end{proof}

\begin{remark}
Item (ii) of Lemma~\ref{lem:properties} shows that the function values $\{L(x^k,y^k)\}_{k\in\N}$ monotonically decrease along the iterations, although $y^{k+1}$ is computed inexactly. An analogous property is also proved for the RAPM algorithm, see \cite[Lemma 3.1(i)]{attouch2010proximal}.

The bound in Item (iv) is a modification of the relative error condition used in the analysis of iterative algorithms under the KL property, which for the RAPM algorithm is written as \cite [Eq. (24)]{attouch2010proximal}
\begin{equation*}
    \|\partial L(x^{k+1},y^{k+1})\|_-\leq b\|(x^{k+1}-x^k,y^{k+1}-y^k)\|, \quad \forall \ k\in\N. 
\end{equation*}
Differently from \cite{attouch2010proximal}, we bound the norm of the subdifferential of the merit function $F$ in place of $L$, at the point $(x^{k+1},\hat{y}^{k+1},{\rho_{k+1}})$ rather than at the actual iterate $(x^{k+1},y^{k+1})$; furthermore, the norm is bounded with a term proportional to $d_k$ in place of the norm $\|(x^{k+1}-x^k,y^{k+1}-y^k)\|$. These differences are due to the fact that the inexactness conditions \eqref{eq:inexact}-\eqref{eq:bound} do not provide any useful information on the subdifferential of $Q^{k+1}$ at the point $y^{k+1}$, while combining the optimality of $\hat{y}^{k+1}$ and conditions \eqref{eq:inexact}-\eqref{eq:bound} we obtain \eqref{bound_sub}. Similar bounds in the literature have been obtained for inexact proximal--gradient methods adopting analogous inexactness conditions, see e.g. \cite{Bonettini-Ochs-Prato-Rebegoldi-2023,Bonettini-Prato-Rebegoldi-2021}.

The bounds in item (iii) are needed to bridge the gap between the actual objective function $L$ and the merit function $F$, whereas the limit in \eqref{eq:limit2} is needed to guarantee that the \REV{KL inequality is applicable to $F$ at the point $(x^k,\hat{y}^k,\rho_k)$ and that each limit point of $\{(x^k,y^k)\}_{k\in\N}$ is stationary for $L$ (see the upcoming proofs of Lemma~\ref{lem:preliminary} and Theorem~\ref{thm:iRAPM_convergence}).}
\end{remark}

Let us define the set of all limit points of the sequence $\{(x^k,y^k)\}_{k\in\N}$, namely
\begin{equation}\label{eq:Xstar}
    \Gamma^*(x^0,y^0) = \{(x^*,y^*)\in A\times B:\exists \ \{k_j\}_{j\in\N}\subseteq \mathbb{N} \ \text{s.t.} \  (x^{k_j},y^{k_j})\rightarrow (x^*,y^*)\}.
\end{equation}
Analogously, $\Xi^*(x^0,y^0,\rho_0)$ denotes the set of all limit points of the sequence $\{(x^k,y^k,\rho_k)\}_{k\in\N}$. By Lemma~\ref{lem:properties}(iii), we have that $\rho_k\rightarrow 0$. Then, we can write $\Xi^*(x^0,y^0,\rho_0)$ as
\begin{equation}\label{eq:Omegastar}
    \Xi^*(x^0,y^0,\rho_0) = \Gamma^*(x^0,y^0)\times \{0\}.
\end{equation}
The following lemma anticipates the main convergence result and contains some basic properties related to the set $\Xi^*(x^0,y^0,\rho_0)$.
\begin{lemma}\label{lem:preliminary}
    Suppose that $\{(x^k,y^k)\}_{k\in\N}$ is bounded. The following statements hold true.
    \begin{itemize}
        \item[(i)] $\Xi^*(x^0,y^0,\rho_0)$ is non-empty and compact.
        \item[(ii)] There exists $L^*\in[0,\infty)$ such that
        \begin{equation*}
            F(x^*,y^*,0) = L(x^*,y^*) = L^*,\quad \forall \ (x^*,y^*)\in \Gamma^*(x^0,y^0).
        \end{equation*}
        \item[(iii)] We have
        \begin{equation*}
            \lim_{k\rightarrow \infty} \operatorname{dist}((x^k,y^k,\rho_k),\Xi^*(x^0,y^0,\rho_0))=\lim_{k\rightarrow \infty}\operatorname{dist}((x^k,\hat{y}^k,\rho_k),\Xi^*(x^0,y^0,\rho_0))=0.
        \end{equation*}
    \end{itemize}
\end{lemma}

{\begin{proof}
(i) Since $\{(x^k,y^k,\rho_k)\}_{k\in\N}$ is bounded, the set $\Xi^*(x^0,y^0,\rho_0)$ is nonempty. As it is a countable intersection of compact sets, the set is also compact \cite[Lemma 5]{Bolte-etal-2014}.

(ii) By Lemma~\ref{lem:properties}(ii) and the continuity of the function $L$ on the set $A\times B$, it follows that there exists $L^*\in\R$ such that
\begin{equation}\label{eq:limL}
    \lim_{k\rightarrow \infty}L(x^k,y^k)=L(x^*,y^*)=L^*,\quad\forall \ (x^*,y^*)\in \Gamma^*(x^0,y^0).
\end{equation}
Then, recalling the definition of $F$ yields the expected result.

(iii) Since $\{(x^k,y^k)\}_{k\in\N}$ is bounded and by definition of $\Gamma^*(x^0,y^0)$ and $\Xi^*(x^0,y^0,\rho_0)$, we have
\begin{equation*}
    \lim_{k\rightarrow \infty}\operatorname{dist}((x^k,y^k),\Gamma^*(x^0,y^0))=\lim_{k\rightarrow \infty}\operatorname{dist}((x^k,y^k,\rho_k),\Xi^*(x^0,y^0,\rho_0))=0.
\end{equation*}
Since  \eqref{eq:limit2} holds and $\{(x^k,\hat{y}^k)\}_{k\in\N}$ is bounded, we also conclude that 
$$
\lim_{k\rightarrow \infty}\operatorname{dist}((x^k,\hat{y}^k,\rho_k),\Xi^*(x^0,y^0,\rho_0))=0.
$$
\end{proof}}
We are now ready to state the main convergence result for the sequence generated by iRAPM, under the assumption that the sequence itself is bounded and $F$ is a KL function. The statement is analogous to the one obtained in Theorem~\ref{thm:conv_APM_regu}(i) for the RAPM algorithm. Before stating and proving our result, we require the following uniformized version of the KL property.

\begin{lemma}\label{lemma:UKL} \cite[Lemma 6]{Bolte-etal-2014}
    Let $F:\R^n\to \bR$ be a proper, lower semicontinuous function and $\Xi\subset \R^n$ a compact set. Suppose that $F$ satisfies the KL property at each point belonging to $\Xi$ and that $F$ is constant over $\Xi$, i.e. $F(z) = \bar{F}\in \R$ for all $z \in \Xi$. Then, there exist $\mu,\upsilon>0$ and a continuous concave function $\phi:[0,\upsilon)\longrightarrow [0,+\infty)$ with $\phi(0) = 0$, $\phi\in C^1(0,\upsilon)$, $\phi'(s) > 0$ for all $s \in (0,\upsilon)$ such that 
    \begin{equation}\label{UKL}
        \phi'(F(z)-\bar{F}) \|\partial F(z)\|_- \geq 1, \quad  \forall \ z\in \bar \Xi 
    \end{equation} 
    where the set $\bar \Xi$ is defined as
    \begin{equation}\label{UKLbound}
        \bar \Xi = \{z\in\R^n: \operatorname{dist}(z,\Xi)<\mu \ \text{and} \ \bar F < F(z) < \bar{F} + \upsilon\}.
    \end{equation}
\end{lemma}

\begin{theorem}\label{thm:iRAPM_convergence}
     Let  $\{(x^k,y^k)\}$ be generated by Algorithm~\ref{algo:APM_regu_inexact}. If the function $F$ given in \eqref{eq:merit} is a KL function,
    then either $\|(x^k,y^k)\|\rightarrow \infty$, or  $\{(x^k,y^k)\}_{k\in\mathbb{N}}$ converges to a stationary point of $L$.
\end{theorem}

\begin{proof}
By Lemma~\ref{lem:properties}, it follows that Algorithm~\ref{algo:APM_regu_inexact} is a special case of the abstract descent algorithm described in \cite[Condition 3]{Bonettini-Ochs-Prato-Rebegoldi-2023}. Then, one could apply \cite[Theorem 5]{Bonettini-Ochs-Prato-Rebegoldi-2023} to deduce the thesis. For the sake of completeness, we report a simplified convergence proof especially tailored for Algorithm~\ref{algo:APM_regu_inexact}.
Assume that the sequence $\{(x^k,y^k)\}_{k\in\N}$ generated by Algorithm~\ref{algo:APM_regu_inexact} is bounded. Let $\{d_k\}_{k\in\mathbb{N}}$ and $\{\rho_k\}_{k\in\N}$ be defined as in Lemma~\ref{lem:properties}, and $L^*$ as in Lemma~\ref{lem:preliminary}. By Lemma~\ref{lem:preliminary}(i)-(ii), the function $F$ is constant and equal to $L^*$ over the compact set $\Xi = \Xi^*(x^0,y^0,\rho_0)$, therefore we can apply Lemma~\ref{lemma:UKL}. {In what follows $\upsilon,\mu,\phi, \bar{\Xi}$ are defined in Lemma~\ref{lemma:UKL}, $a,b$ in Lemma~\ref{lem:properties}}.

Lemma~\ref{lem:properties}(iii), Lemma~\ref{lem:preliminary}(iii) and equation \eqref{eq:limL} guarantee the existence of a positive integer $k_0$ such that
\begin{equation}\label{upsilon}
    L(x^k,y^k)+\gamma_k < L^* + \upsilon, \quad \operatorname{dist}( (x^k,\hat{y}^k,\rho_k),\Xi^*(x^0,y^0,\rho_0) ) < \mu, \quad \forall \ k\geq k_0.
\end{equation}
Without loss of generality, up to a translation of the iteration index, we can assume $k_0=0$.

Since \REV{$\{L(x^k,y^k)\}_{k\in\mathbb{N}}$ is a monotone nonincreasing sequence and $L^*$ is its limit (see Lemma~\ref{lem:properties}(iii) and \eqref{eq:limL}), } we have $L(x^k,y^k)\geq L^*$ for all $k\geq 0$, which implies that $\phi$ can be applied to both values $L(x^k,y^k)-L^*$ and $L(x^{k+1},y^{k+1})-L^*$. Then, one can define the following quantity
\begin{equation}\label{eq:phik}
    \phi_k = \frac{b}{a}(\phi(L(x^k,y^k)-L^*)- \phi(L(x^{k+1},y^{(k+1)})-L^*)), \quad \forall \ k\geq 0.
\end{equation}
Since $\phi$ is a monotone increasing function and $\{L(x^k,y^k)\}_{k\in\mathbb{N}}$ a monotone nonincreasing sequence, we have $\phi_k\geq 0$, for all $k\geq 0$. Let us show that the following inequality holds:
\begin{equation}\label{th2:eq1}
    2d_k\leq \phi_k+d_{k-1}, \quad \forall \ k\geq 1.
\end{equation}
If $d_k=0$, inequality \eqref{th2:eq1} holds trivially. Otherwise, for any iteration index $k\geq 1$ such that $d_k>0$ and using Lemma~\ref{lem:properties}(ii)-(iii), we can write
\begin{equation*}
    L^*\leq L(x^{k+1},y^{k+1})< L(x^k,y^k)\leq F(x^k,\hat{y}^k,\rho_k) \leq L(x^{k-1},{y}^{k-1}) + \gamma_{k-1},
\end{equation*}
which, in view of \eqref{upsilon}, implies that $(x^{k},\hat{y}^{k},\rho_k)\in \bar \Xi$.
Then, combining the KL inequality related to the function $F$ at $(x^{k},\hat{y}^{k},\rho_k)$ with Lemma~\ref{lem:properties}(iv) yields:
\begin{equation}\label{eq:same_old_ine}   \phi'(F(x^{k},\hat{y}^{k},\rho_k)-L^*)\geq \frac{1}{\|\partial F(x^{k},\hat{y}^{k},\rho_k)\|_-}\geq \frac{1}{b d_{k-1}}.
\end{equation}
By employing the concavity of the function $\phi$, we obtain
\begin{align*}
\phi(L(x^k,y^k)-L^*)-\phi(L(x^{k+1},y^{k+1})-L^*)&\geq \phi'(L(x^k,y^k)-L^*)(L(x^k,y^k)-L(x^{k+1},y^{k+1})) \\ 
&\geq  \frac{ad_k^2}{bd_{k-1}},
\end{align*}
{where the latter inequality follows by applying Lemma~\ref{lem:properties}(ii)-(iii), the monotonicity of $\phi'$ and \eqref{eq:same_old_ine}.} 
By recalling the definition of $\phi_k$ in \eqref{eq:phik}, the above inequality implies $d_k^2\leq \phi_kd_{k-1}$.
Now taking the square root of both sides and using the inequality $2\sqrt{uv}\leq u+v$ on the right-hand side we get inequality \eqref{th2:eq1}. By summing \eqref{th2:eq1} over $k=0,\ldots,K$ and taking the limit for $K\rightarrow \infty$, following the same arguments as in \cite[Theorem 5]{Bonettini-Ochs-Prato-Rebegoldi-2023}, we easily get
\begin{equation}\label{eq:summable_dk}
    \sum_{k=0}^{\infty}d_k < \infty.
\end{equation} 
Finally, by combining \eqref{eq:summable_dk} and \eqref{eq:ineq_dk}, we conclude that $
\sum_{k=0}^{\infty}\|(x^{k+1}-x^k,y^{k+1}-y^k)\|<\infty$. Hence the sequence $\{(x^k,y^k)\}_{k\in\N}$ converges to a point $(x^*,y^*)\in A\times B$.

Since $\{(x^k,y^k)\}_{k\in\N}$ converges to $(x^*,y^*)$ and $\{\rho_k\}_{k\in\N}$ converges to $0$ (due to Lemma~\ref{lem:properties}(iii)), we have $(x^k,y^k,\rho_k)\rightarrow (x^*,y^*,0)$; furthermore, thanks to \eqref{eq:limit2} and the continuity of $F$ in $A\times B$, it also follows that $(x^k,\hat{y}^k,\rho_k)\rightarrow (x^*,y^*,0)$ and $F(x^k,\hat{y}^k,\rho_k)\rightarrow F(x^*,y^*,0)$. Finally, by summing the inequality  \eqref{bound_sub} for $k=0,\ldots,K$, taking its limit for $K\rightarrow \infty$ and recalling that $\{d_k\}_{k\in\N}$ is summable (see \eqref{eq:summable_dk}), we get
\begin{equation*}
    \sum_{k=0}^\infty\|\partial F(x^k,\hat{y}^k,\rho_k)\|_{-}<\infty \quad \Rightarrow \quad \lim_{k\rightarrow \infty}\|\partial F(x^k,\hat{y}^k,\rho_k)\|_{-}=0.
\end{equation*}
To summarize, the sequence $\{(x^k,\hat{y}^k,\rho_k)\}_{k\in\N}$ satisfies the following limits:
\ste{\begin{equation*}
    \lim_{k\rightarrow \infty}(x^k,\hat{y}^k,\rho_k)= (x^*,y^*,0), \quad \lim_{k\rightarrow \infty}F(x^k,\hat{y}^k,\rho_k)= F(x^*,y^*,0),
    \end{equation*}
    and 
    \begin{equation*}
 \lim_{k\rightarrow \infty}\|\partial F(x^k,\hat{y}^k,\rho_k)\|_{-}=0.
\end{equation*}}
Then, it satisfies the hypotheses of \cite[Lemma 2.1]{Frankel-etal-2015}, which ensures that $0\in \partial F(x^*,y^*,0)$. Since $\partial F(x^*,y^*,0)=\partial L(x^*,y^*)\times \{0\}$, it follows  $0\in\partial L(x^*,y^*)$, i.e., $(x^*,y^*)$ is a stationary point of $L$.
\end{proof}

Inspired by \cite{attouch20009proximal}, we now investigate the convergence rate of $\{(x^k,y^k)\}_{k\in\N}$ when the merit function $F$ given in \eqref{eq:merit} satisfies the KL property at the limit point $(x^*,y^*)$ with $\phi(t)=ct^{1-\theta}$, with $c>0, \ \theta\in [0,1[$. {Similarly to the result obtained for the RAPM algorithm (see Theorem~\ref{thm:conv_APM_regu}(iii)), the convergence rate of iRAPM depends on the value of $\theta$ appearing in the desingularizing function $\phi$.
}

\begin{theorem}\label{thm:iRAPM_rate}
    Suppose that $\{(x^k,y^k)\}_{k\in\N}$ converges to a point $(x^*,y^*)\in A\times B$ . Additionally, assume that $F$ defined in \eqref{eq:merit} satisfies the KL property at $(x^*,y^*)$ with $\phi(t)=ct^{1-\theta}$, with $c>0, \ \theta\in [0,1)$. Then the following convergence rates hold.
    \begin{itemize}
        \item[(i)] If $\theta=0$, the sequence $(x^k,y^k)$ terminates in a finite number of iterations.
        \item[(ii)]If $\theta\in]0,\frac{1}{2}]$, there exists $C>0 $ and $\tau\in (0,1)$ such that
        \begin{equation}\label{eq:linear}
        \|(x^k,y^k)-(x^*,y^*)\|\leq C \tau^k.
        \end{equation}
        \item[(iii)] If $\theta\in]\frac{1}{2},1[$, there exists $C>0$ such that
        \begin{equation}\label{eq:sublinear}
        \|(x^k,y^k)-(x^*,y^*)\|\leq Ck^{-\frac{1-\theta}{2\theta-1}}.
        \end{equation}
    \end{itemize}
\end{theorem}

\begin{proof}
{(i) If $\theta=0$, then $\phi(t) = ct$ and $\phi'(t) = c$ for $t\in [0,\upsilon)$. By contradiction, assume that there exists an infinite subset $K\subseteq \N$ such that $(x^{k+1},y^{k+1})\neq (x^k,y^k)$ for all $k\in K$. From Lemma~\ref{lem:properties}(i), this means that $d_k> 0$ for all $k\in K$, where $d_k$ is given in \eqref{eq:dk}. By using the same reasoning as in the proof of Theorem~\ref{thm:iRAPM_convergence}, we conclude that $(x^{k},\hat{y}^{k},\rho_k)\in \bar \Xi$ for all $k\in K$. Then, for all $k\in K$, the uniformized KL property (i.e. Lemma~\ref{lemma:UKL}) applies as follows 
\begin{equation}\label{eq:sub_lower_bound}
    \|\partial F(x^k,\hat{y}^k,\rho_k)\|_{-}\geq \frac{1}{\phi'(F(x^k,\hat{y}^k,\rho_k)-L^*)}=\frac{1}{c}, \quad \forall \ k\in K.
\end{equation}
By combining \eqref{eq:sub_lower_bound} with \eqref{bound_sub}, we get $ d_{k-1}\geq \frac{1}{bc} $ $ \forall \ k\in K$ and Lemma~\ref{lem:properties}(ii) yields
\begin{equation*}
    \frac{a}{b^2c^2}\leq L(x^{k-1},y^{k-1})-L(x^{k},y^{k}),\quad \forall \ k\in K,
\end{equation*}
which is absurd since the right-hand side of the inequality is converging to zero. Then item (i) follows.}

(ii)-(iii) If $\theta\in (0,1)$, then let $(x^*,y^*)\in A\times B$ be the unique limit point of $\{(x^k,y^k)\}_{k\in\N}$. By summing inequality \eqref{th2:eq1} for all indexes $k,k+1,\ldots,K$ with $k<K$, we get
\begin{equation}\label{eq:instrumental1}
    2\sum_{\ell = k}^Kd_\ell \leq \sum_{\ell = k}^K\phi_\ell+\sum_{\ell = k-1}^{K-1}d_\ell.
\end{equation}
Due to the definition of $\phi_k$ in \eqref{eq:phik} and the nonnegative sign of the desingularizing function $\phi$, one can also write the inequality below:
\begin{equation}\label{eq:instrumental3}
    \sum_{\ell=k}^K\phi_\ell \leq \frac{b}{a}\phi(L(x^k,y^k)-L^*).
\end{equation}
By applying inequality \eqref{eq:instrumental3} to \eqref{eq:instrumental1}, we get
\begin{align*}
    2\sum_{\ell=k}^Kd_\ell \leq \frac{b}{a}\phi(L(x^k,y^k)-L^*)+\sum_{\ell = k-1}^{K-1}d_\ell. 
\end{align*}
By passing to the limit for $K\rightarrow \infty$ and letting  
\begin{equation*}
\Delta_k = \sum_{\ell = k}^{\infty} d_\ell, \quad \forall \ k\in\N,
\end{equation*}
we finally obtain
\begin{equation}\label{eq:crucial_for_rate}
    \Delta_k \leq \frac{b}{a}\phi(L(x^k,y^k)-L^*)+\Delta_{k-1}-\Delta_k,  \quad \forall \ k\in\mathbb{N}.
\end{equation}
Let $\bar{\Xi}\subseteq \R^n$ be defined as in the proof of Theorem~\ref{thm:iRAPM_convergence}. Since $(x^{k},\hat{y}^{k},\rho_k)\rightarrow (x^*,y^*,0)$, we can assume without loss of generality that $(x^{k},\hat{y}^{k},\rho_k)\in\bar{\Xi}$ for all $k\geq 0$. {Then,  the following inequalities hold}
\begin{equation*}
    \phi'(L(x^k,y^k)-L^*)\geq \phi'(F(x^{k},\hat{y}^{k},\rho_k)-L^*)\geq \frac{1}{\|\partial F(x^{k},\hat{y}^{k},\rho_k)\|_{-}},
\end{equation*}
where the first inequality follows from \eqref{eq:limit3} and the monotonicity of $\phi'$, whereas the second one is the application of Lemma~\ref{lemma:UKL} at the point $(x^{k},\hat{y}^{k},\rho_k)$. Recalling that $\phi(t)=ct^{1-\theta}$ yields
\begin{equation*}
    c(1-\theta)(L(x^k,y^k)-L^*)^{-\theta}\|\partial F(x^{k},\hat{y}^{k},\rho_k)\|_{-}\geq 1,
\end{equation*}
and by rearranging terms in the previous inequality and applying Lemma~\ref{lem:properties}(iv), we get
\begin{align*}
    (L(x^k,y^k)-L^*)^{\theta}&\leq c(1-\theta)\|\partial F(x^{k},\hat{y}^{k},\rho_k)\|_{-}\leq bc(1-\theta)d_{k-1}
    =bc(1-\theta)(\Delta_{k-1}-\Delta_k).
\end{align*}
Hence, the previous inequality guarantees the existence of a positive constant $q>0$ such that
\begin{equation*}
    \phi(L(x^k,y^k)-L^*)\leq q(\Delta_{k-1}-\Delta_k)^{\frac{1-\theta}{\theta}}.
\end{equation*}
Then, inequality  \eqref{eq:crucial_for_rate} leads to
\begin{equation} \label{eq:crucial1_for_rate}
    \Delta_k\leq \frac{bq}{a}(\Delta_{k-1}-\Delta_k)^{\frac{1-\theta}{\theta}}+\Delta_{k-1}-\Delta_k, \quad \forall \ k\in\mathbb{N}.
\end{equation}
Any nonnegative sequence $\{\Delta_k\}_{k\in\N}$ complying with an inequality of form \eqref{eq:crucial1_for_rate} converges to zero with rate either $\mathcal{O}(\tau^k)$, $\tau\in (0,1)$ or $\mathcal{O}(k^{-(1-\theta)/(2\theta-1)})$, depending on the value of $\theta$, see the proof of \cite[Theorem 2]{attouch20009proximal}. {By Lemma~\ref{lem:properties}(i) and the triangular inequality, it is immediate to verify that $\Delta_k\geq \sum_{\ell = k}^{\infty}\|(x^{\ell+1}-x^\ell,y^{\ell+1}-y^{\ell})\|\geq \|(x^{k}-x^*,y^{k}-y^*)\|$. Hence, the sequence $\{\|(x^{k}-x^*,y^{k}-y^{*})\|\}_{k\in\N}$ shares the same convergence rates as $\{\Delta_k\}_{k\in\N}$ and the proof is complete.}
\end{proof}

%%%%%%%%%%%%%%%%%%%%%%%%%%%%%%%%%%%%%%%%%%%
% Application to Affine Rank Minimization %
%%%%%%%%%%%%%%%%%%%%%%%%%%%%%%%%%%%%%%%%%%%
\section{Application to affine rank minimization}\label{sec:application}

In this section, we focus on problem \eqref{affine_rank_min}. Assuming that the rank $r$ of the sought matrix $Z$ is given, the affine rank minimization problem can be formulated as
\begin{equation}\label{eq:feasibility}
    \text{find }Z\in\R^{n_1\times n_2}\text{ such that }Z\in \setNorm\cap \setRank, 
\end{equation}
where
\begin{align}\label{eq:constraints}
    \setNorm &= \{Z\in \mathbb{R}^{n_1\times n_2}: \ {\cal A}(Z) = b\},\\
    \setRank &= \{Z\in\mathbb{R}^{n_1\times n_2}: \rank{Z}\leq r\}.
\end{align}
Since $\setNorm$ is an affine subspace, the projection onto it reduces to a least squares problem; however, in applications such as Matrix Completion \cite{candes2012exact}, it consists in trivial element-wise operations on the matrix. Then, in the following, we consider the projection onto $\setNorm$ as known in closed form and cheap to compute. Regarding the projection onto the set $\setRank$, we can rely on the truncated SVD of the matrix we would like to project, as the well-known result by Eckart and Young \cite[Theorem 2.4.8]{golub2013matrix} states that, given a matrix $G \in \R^{n_1 \times n_2}$ with $r < \rank{G} \le \min \{ n_1, n_2 \}$, the nearest matrix of rank $r$ to $G$ (with respect to  $\| \cdot \| $) is given by the $r$-truncated SVD; using the notation given in \eqref{eq:svd_part} we can write
\begin{equation}
    \label{eq:proj_rank}
    \proj_\setRank(G) = \Tsvd{G} = U_1 \Sigma_1 V_1^T.
\end{equation}
We now  establish the prox-regularity of $C_r$ at any matrix of rank $r$ and the semialgebraicity of $\mathcal{C} $ and $\mathcal{C}_r$. These properties allow us  to  prove  the  convergence  of APM, RAPM and iRAPM  when applied to affine rank minimization problems. 

\begin{lemma} \label{lemma:prox-regular}\cite[Proposition 3.8]{luke2013prox} 
    The set $\setRank$ is prox-regular at all points $\bar{X}\in\R^{n_1\times n_2}$ with $\rank{\bar{X}} = r $.
\end{lemma}

\begin{lemma}\label{lemma_semialgebrici}
    The sets $\setNorm$ and $\setRank$ are closed semialgebraic subsets of $\R^{n_1\times n_2}$.
\end{lemma}

\begin{proof}
Note that the set $\setNorm$ is an affine subspace, as it is defined by a finite number of affine equations. Hence, $\setNorm$ complies with the definition of semialgebraic set of $\R^{n_1\times n_2}$ given in equation \eqref{eq:semialgebraic}.
Likewise,  {given $I\subseteq \{1,\ldots,n_1\}, \ J\subseteq \{1,\ldots,n_2\}$, we denote with $Z_{IJ}$ the submatrix of $Z$ obtained by keeping the rows and the columns with indexes in $I$ and $J$, respectively.} Observe that 
$$
    \rank Z\leq r \quad\Leftrightarrow \quad \det(Z_{IJ})=0, \ \forall I\subseteq \{1,\ldots,n_1\}, \ J\subseteq \{1,\ldots,n_2\}, \,\mbox{s.t } |I|=|J|=r+1. 
$$
The function $\det:\mathbb{R}^{(r+1)\times (r+1)}\rightarrow \mathbb{R}$ is polynomial with respect to the elements of $Z$. Therefore we can conclude that $\mathcal{C}_r$ is a semialgebraic set, as it is defined by a finite number of polynomial equations.
\end{proof}

\subsection{Convergence Results}
We now specialize the convergence properties of {APM and RAPM} described in Section 2 and {iRAPM proposed in Section 3} when applied to the affine rank minimization problem \eqref{eq:feasibility}. We first need to adjust the notation. {The Euclidean norm naturally becomes the Frobenius norm,} the sets $A$ and $B$ are now $\setNorm$ and $\setRank$, respectively, the alternating projection methods generate two sequences $\{X^k\}_{k\in\mathbb{N}}, \ \{Y^k\}_{k\in\mathbb{N}}$ of matrices and the ``regularized'' points are defined as follows:
\begin{equation}\label{eq:Ykreg}
    \Xck = \frac{1}{1 + \lambda_k} X^k + \frac{\lambda_k}{1 + \lambda_k} Y^k,  \quad \Yck = \frac{\mu_k}{1 + \mu_k} X^{k+1} + \frac{1}{1 + \mu_k} Y^k.
\end{equation}
Then, the function $L$ in \eqref{eq:problem_L} takes the form
\[
    L(X,Y) = \frac{1}{2} \| X-Y \|^2 + \iota_\setNorm (X) + \iota_\setRank (Y).
\]
{From the analysis carried out in \cite{Noll-etal-2016} (see Theorem~\ref{thm:conv_APM}) and the analytical properties of the sets $\mathcal{C}$ and $\mathcal{C}_r$,} we easily derive the following result for APM applied to the feasibility problem \eqref{eq:feasibility}.

\begin{corollary}\label{th:APM_MC}
    {Suppose that there exists $\bar{X}\in\setNorm \cap \setRank$ such that $\rank{\bar{X}}=r$}. Let $\{X^k\}_{k\in\mathbb{N}}\subseteq \R^{n_1\times n_2}$ and $\{Y^k\}_{k\in\mathbb{N}}\subseteq \R^{n_1\times n_2}$ be the sequences generated by APM with $A=\setNorm$ and $B=\setRank$. Then $\{X^k\}_{k\in\mathbb{N}}$ and $\{Y^k\}_{k\in\mathbb{N}}$ locally converge to a matrix $Z\in\R^{n_1\times n_2}$ with $Z\in \setNorm \cap \setRank $ \REV{with rates $\|X^k-Z\|=\mathcal{O}(k^{-\frac{2-\omega}{2\omega}})$ and $\|Y^k-Z\|=\mathcal{O}(k^{-\frac{2-\omega}{2\omega}})$ for some $\omega\in (0,2)$.} 
\end{corollary}

\begin{proof}
{Lemma~\ref{lemma_semialgebrici} ensures that  $\setNorm$ and $\setRank$ are closed semialgebraic subsets of $\R^{n_1\times n_2}$, and every semialgebraic set is subanalytic, see e.g. \cite[p. 1208]{Bolte2007}. Since $\setNorm,\setRank$ are closed and subanalytic, it follows that they intersect separably \REV{with $\omega\in (0,2)$} \cite[Theorem 3]{Noll-etal-2016}.} Being $\setRank$ prox-regular at $\bar{X}$ by Lemma~\ref{lemma:prox-regular}, it follows that $\setRank$ is $\sigma-$H\"older regular with respect to $\setNorm$ at $\bar{X}$ for any $\sigma\in[0,1)$ and any arbitrarily small constant $c$ \cite[Corollary 3]{Noll-etal-2016}. Then, the hypotheses of Theorem~\ref{thm:conv_APM} hold, and the thesis follows.
\end{proof}

Likewise, the following Corollary follows for RAPM when applied to \eqref{eq:feasibility}.

\begin{corollary}\label{th:RAPM_MC}
    Let $\{X^k\}_{k\in\mathbb{N}}\subseteq \R^{n_1\times n_2}$ and $\{Y^k\}_{k\in\mathbb{N}}\subseteq \R^{n_1\times n_2}$ be the sequences generated by RAPM (Algorithm~\ref{algo:APM_regu}) with $A=\setNorm$, $B=\setRank$.
    \begin{itemize}
        \item[(i)] Either $\|(X^k,Y^k)\|\rightarrow \infty$, or the sequence $\{(X^k,Y^k)\}_{k\in\mathbb{N}}$ converges to a stationary point of the function $L$.
        \item[(ii)] Let $(\bar{X},\bar{Y})$ be any point such that $\|\bar{X}-\bar{Y}\|=\min\{\|X-Y\|: \ X\in \setNorm, \ Y\in \setRank\}$. If $(X^0,Y^0)$ is sufficiently close to $(\bar{X},\bar{Y})$, then the whole sequence $\{(X^k,Y^k)\}_{k\in\mathbb{N}}$ converges to a point $(X_{\infty},Y_{\infty})$ such that $\|X_{\infty}-Y_{\infty}\| =\min\{\|X-Y\|: \ X\in \setNorm, \ Y\in \setRank\}$.
        \item[(iii)]{Suppose that the sequence $\{(X^k,Y^k)\}_{k\in\N}$ converges to $(X^*,Y^*)\in \setNorm \times \setRank$. \REV{Then, $\{(X^k,Y^k)\}_{k\in\N}$ satisfies one of the following: it terminates in a finite number of iterations; it converges with a linear rate of the form \eqref{eq:linear}; it converges with a sublinear rate of the form \eqref{eq:sublinear}.}}
    \end{itemize}
\end{corollary}

\begin{proof}
The function $L$ is given by the sum of three semialgebraic functions, which implies that $L$ is itself semialgebraic \cite[Example 2]{Bolte-etal-2014}. \REV{Furthermore, $L$ is a semialgebraic function that is continuous on its closed domain, which implies that it satisfies the KL property at any point with $\phi(t)=ct^{1-\theta}$, being $c>0$ and $\theta\in[0,1)$ two parameters depending on the given point \cite[Theorem 3.1]{Bolte2007}.} Hence, the thesis follows by applying Theorem~\ref{thm:conv_APM_regu} with the choice $A=\setNorm$, $B=\setRank$.
\end{proof}

{In analogy with Corollary~\ref{th:APM_MC}, we can be more precise on the convergence rate of RAPM for affine rank minimization problems if we assume that the iterate sequence converges to the intersection of the two sets. In particular, we state and prove the following Corollary, which is, to the best of our knowledge, new, and whose proof employs a result on inf-projection taken from \cite{Yu2022} in combination with a property holding for subanalytic sets \cite{Noll-etal-2016}. 

\begin{corollary}\label{thm:RAPM_precise_rate}
    Suppose that the sequence $\{(X^k,Y^k)\}_{k\in\N}$ generated by RAPM (Algorithm~\ref{algo:APM_regu}) converges to $(X^*,X^*)$ with $X^*\in \mathcal{C}\cap \mathcal{C}_r$. Then, there exists $\theta\in (\frac{1}{2},1)$ such that $\|(X^k,Y^k)-(X^*,X^*)\|=\mathcal{O}(k^{-\frac{1-\theta}{2\theta-1}})$.
\end{corollary}

\begin{proof}
We let $m:\R^{n_1\times n_2}\rightarrow \R\cup\{\infty\}$ be defined as the inf-projection
\begin{align}\label{eq:f}
    m(X) &= \frac{1}{2}\min_{Y\in \setRank}\|X-Y\|^2+\iota_\setNorm(X) \nonumber\\
    &= \min_{Y\in\R^{n_1\times n_2}}\left(\frac{1}{2}\|X-Y\|^2+\iota_\setNorm(X)+\iota_\setRank(Y)\right)=\min_{Y\in\R^{n_1\times n_2}}L(X,Y).
\end{align}
Let us also denote with $M$ the following set-valued function
\begin{equation}\label{eq:m}
    M(X) = \underset{Y\in \R^{n_1\times n_2}}{\operatorname{argmin}} \ L(X,Y).
\end{equation}
We observe that the hypotheses of \cite[Theorem 3.1]{Yu2022} hold, indeed:
\begin{itemize}
    \item If $\bar{X}\in \setNorm \cap \setRank$, then $M(\bar{X})$ is single-valued and $M(\bar{X})=\{\bar{X}\}$. 
    \item $\partial L(X,Y)\neq \emptyset, \ \forall \ X\in \setNorm, \ \forall \ Y\in \setRank.$
    \item The function $L$ is lower-bounded with respect to $Y$ locally uniformly in $X$, which means that, for all $\bar{X}\in\R^{n_1\times n_2}$ and for all $\alpha\in\R$, the lower-level sets $\mathcal{L}_{\alpha,X}=\{Y\in\R^{n_1\times n_2}: \ L(X,Y)\leq \alpha\}$ are uniformly bounded in a neighbourhood of $\bar{X}$. Indeed, note that
    \begin{equation*}
        \|Y\|\leq \|Y-X\|+\|X\|\leq \sqrt{2\alpha}+\|X\|, \quad \forall \ Y\in \mathcal{L}_{\alpha,X}.
    \end{equation*}
    Then, if $X\in \ste{\mathcal{B}(\bar{X},\epsilon)}$, namely $X$ belongs to the ball of center $\bar{X}$ and radius $\epsilon$, it follows that the sets $\mathcal{L}_{\alpha,X}$ are uniformly bounded by $C=\sqrt{2\alpha}+\|\bar{X}\|+\epsilon$ for all $X\in B(\bar{X},\epsilon)$.
\end{itemize}
Therefore, we can apply \cite[Theorem 3.1]{Yu2022}, which states that, if $\bar{X}\in \mathcal{C}\cap \mathcal{C}_r$ and $L$ satisfies the KL property at the point $(\bar{X},\bar{X})$ with desingularizing function $\phi(t)=Ct^{1-\theta}$, then the inf-projection $m$ defined in \eqref{eq:f} satisfies the KL property at $(\bar{X},\bar{X})$ with the same exponent $\theta\in[0,1)$. 

Now, we recall that $\mathcal{C},\mathcal{C}_r$ are both subanalytic subsets of $\R^{n_1\times n_2}$. Then, from \cite[Lemma 3]{Noll-etal-2016}, it follows that the KL exponent $\theta$ of $m$ at the point $(\bar{X},\bar{X})$ must be strictly greater than $\frac{1}{2}$. Therefore, we can conclude that the KL exponent of $L$ at $(\bar{X},\bar{X})$ is greater than $\frac{1}{2}$, otherwise there would be a contradiction due to \cite[Theorem 3.1]{Yu2022}. At this point, the thesis follows by applying Theorem~\ref{thm:conv_APM_regu}(iii) with the choice $A=\setNorm$, $B=\setRank$\ste{, and $(x^*,y^*) = (X^*,X^*)$}.
\end{proof}}

{
For our proposed iRAPM applied to \eqref{eq:feasibility}, we can prove the following convergence result.}

\begin{corollary}\label{th:iRAPM_MC}
    Let $\{X^k\}_{k\in\mathbb{N}}\subseteq \R^{n_1\times n_2}$ and $\{Y^k\}_{k\in\mathbb{N}}\subseteq \R^{n_1\times n_2}$ be the sequences generated by iRAPM (Algorithm~\ref{algo:APM_regu_inexact}) with $A=\setNorm$, $B=\setRank$. Then either $\|(X^k,Y^k)\|\rightarrow \infty$, or the sequence $\{(X^k,Y^k)\}_{k\in\mathbb{N}}$ converges to a stationary point $(X^*,Y^*)\in\mathcal{C} \times \mathcal{C}_r$ of the function $L$. {In the latter case, the sequence satisfies one of the following: it terminates in a finite number of iterations; it converges with a linear rate of the form \eqref{eq:linear}; it converges with a sublinear rate of the form \eqref{eq:sublinear}.}
\end{corollary}

\begin{proof}
Since the merit function 
\begin{equation}\label{eq:FXYR}
F(X,Y,\rho)=L(X,Y)+\frac{1}{2}\rho^2,
\end{equation}
is the sum of two semialgebraic functions, it follows that $F$ is itself semialgebraic; \REV{furthermore, $F$ is continuous on its closed domain. Hence, by \cite[Theorem 3.1]{Bolte2007}, $F$ satisfies the KL property at any point with $\phi(t)=ct^{1-\theta}$, being $c>0$ and $\theta\in[0,1)$.} Then the thesis follows by applying Theorems~\ref{thm:iRAPM_convergence}-\ref{thm:iRAPM_rate} to Algorithm~\ref{algo:APM_regu_inexact} with the choice $A=\setNorm$, $B=\setRank$.
\end{proof}

{As done for RAPM, we also specify the convergence rate of iRAPM when the iterates sequence converges to an intersection point.

\begin{corollary}\label{thm:iRAPM_precise_rate}
Suppose that the sequence $\{(X^k,Y^k)\}_{k\in\N}$ generated by iRAPM (Algorithm~\ref{algo:APM_regu_inexact}) converges to $(X^*,X^*)$ with $X^*\in \mathcal{C}\cap \mathcal{C}_r$. Then, there exists $\theta\in (\frac{1}{2},1)$ such that $\|(X^k,Y^k)-(X^*,X^*)\|=\mathcal{O}(k^{-\frac{1-\theta}{2\theta-1}})$.
\end{corollary}

\begin{proof}
Letting $Z_1 = (X,Y)$, $Z_2 = \rho$ and $Z=(Z_1,Z_2)$, the function $F$ given in \eqref{eq:FXYR} can be seen as a block separable sum of the form $F(Z) = F_1(Z_1)+F_2(Z_2)$, where $F_1=L$ and $F_2(Z_2)=\frac{1}{2}Z_2^2$. Note that $F_1$ satisfies the KL property at any point $\bar{X}\in\mathcal{C}\cap\mathcal{C}_r$ with exponent $\theta>\frac{1}{2}$ (see proof of Theorem~\ref{thm:RAPM_precise_rate}), whereas $F_2$ has the same property with exponent $\frac{1}{2}$. By \cite[Theorem 3.3]{Li2018}, we conclude that $F$ has the KL property at any point $\bar{X}\in\mathcal{C}\cap\mathcal{C}_r$ with exponent $\theta>\frac{1}{2}$. Finally, the thesis follows by applying Theorem~\ref{thm:iRAPM_rate}(iii) to Algorithm~\ref{algo:APM_regu_inexact} with the choice $A=\setNorm$, $B=\setRank$\ste{, and $(x^*,y^*) = (X^*,X^*)$}.     
\end{proof}}

\subsection{Inexact Projection} \label{sec:Lanczos} 
\REV{In the following, we fix $k\in\mathbb{N}$ and let $Y_{reg}^k$ be the matrix to be projected and defined in \eqref{eq:Ykreg}, $\hat{Y}^{k+1}$ any matrix belonging to the projection set $\mathcal{P}_{\mathcal{C}_r}(Y_{reg}^k)$, and $Y_{k+1}$ the inexact projection satisfying \eqref{eq:inexact}-\eqref{eq:bound}.}

{As outlined in Section~\ref{sec:practical}, the accuracy requirements in \eqref{eq:inexact}-\eqref{eq:bound} can be enforced without the knowledge of the exact projection provided that it is possible to define two sequences of scalars $c^\ell, a^\ell$ and a sequence of matrices $W^\ell$ satisfying \eqref{eq:cond_1}-\eqref{eq:cond_3} {and \eqref{eq:cond_4}}. Then,  the inexact projection is given by  $W^{\bar \ell}$, where $\bar \ell$ is the first iteration such that  \eqref{eq:stopping_final} is met.
In this section, by exploiting  the Lanczos process \cite{saad2011numerical,larsen1998lanczos}, we are going to show how to  compute an approximation of the  $r$-truncated SVD of $\Yck$ that represents an inexact projection satisfying
\eqref{eq:inexact}-\eqref{eq:bound}.
We first  briefly describe how the Lanczos method works. }

\subsubsection{The Lanczos method}
We employ the Lanczos procedure for approximate the $r$ largest singular values of the matrix $\Yck\in\R^{n_1\times n_2}$. Given an integer $\ell$ with $r \le \ell \le \min\{n_1, n_2\}$, such a procedure uses Lanczos bidiagonalization, also known as Paige-Saunders bidiagonalization \cite{paige1982lsqr}, to find two orthonormal matrices $P_{\ell+1} \in \R^{n_1 \times (\ell +1)}$, $Q_\ell \in \R^{n_2 \times \ell}$, such that
\[
    P_{\ell+1}^T \Yck Q_\ell = B_\ell = 
    \left(
    \begin{array}{c@{\hspace{.5em}}c@{\hspace{.5em}}c@{\hspace{.5em}}c}
        \alpha_1 &           &        &  \\
        \beta_2  & \alpha_2  &        &  \\
                 & \beta_3   & \smash{\ddots} &  \\
                 &           & \smash{\ddots} & \alpha_\ell \\
                 &           &        & \beta_{\ell+1}
    \end{array}
    \right) \in \R^{(\ell+1) \times \ell}.
\]
If we consider the columns partitioning \REV{$P_{\ell+1} = [p_1, \dots, p_{\ell + 1}]$} and $Q_\ell = [q_1, \dots, q_{\ell}]$, the column vectors $p_j, q_j$ are the so-called \textit{Lanczos vectors} and  can be built by recurrence, starting from two vectors $q_0, \ p_1$;
indeed the following recurrence holds for $\ell = 1, 2,\dots$:
\begin{align}
    \alpha_\ell q_\ell &=(\Yck)^T p_\ell - \beta_\ell q_{\ell-1} \label{eq:Lanczos_recurrence_vector_1} \\
    \beta_{\ell+1} p_{\ell+1} &= \Yck q_\ell - \alpha_\ell p_\ell, \label{eq:Lanczos_recurrence_vector_2}
\end{align}
where $\alpha_\ell, \ \beta_\ell$ are chosen such that the Lanczos vectors have unitary norm. Moreover, we can rewrite the recurrence in matrix form as:
\begin{align}
    \Yck Q_\ell &= P_{\ell + 1} B_\ell \label{eq:Lanczos_recurrence_matrix_1} \\
    (\Yck)^T P_{\ell +1} &= Q_\ell B_\ell^T + \alpha_{\ell+1}q_{\ell+1}e_{\ell+1}^T, \label{eq:Lanczos_recurrence_matrix_2}
\end{align}
where $e_{\ell+1}$ is the $(\ell+1$)-th vector of the canonical basis.
Given the bidiagonal matrix $B_\ell$ at hand,  its SVD $B_\ell = U_B \Sigma_B V_B^T$ can be cheaply computed and   the following matrix $ \ML_\ell\in\R^{n_1\times n_2}$ is \REV{defined}%obtained
\begin{equation} \label{svd_Z}
    \cZ_\ell = \REV{P_{\ell+1} B_{\ell} Q_{\ell}^T}= \underbrace{P_{\ell+1} U_B}_{\tilde U_\ell} \underbrace{\Sigma_B}_{\tilde \Sigma_\ell} \underbrace{V_B^T Q_\ell^T}_{\tilde V_\ell^T}=\tilde U_\ell \tilde \Sigma_\ell \tilde V_\ell^T,
\end{equation}
where $\tilde \Sigma_\ell\REV{\in\mathbb{R}^{(\ell+1)\times \ell}}$ is a rectangular diagonal matrix whose diagonal entries $\tilde \sigma_{1,\ell}, \dots, \tilde \sigma_{\ell,\ell}$ can be seen as approximations of the largest $\ell$ singular values \REV{$\sigma_1,\dots,\sigma_\ell$} of $\Yck$. \REV{Note also that the first $\ell$ singular values of $B_\ell$ and $G_\ell$ are the same and $G_\ell$ has $\min\{n_1,n_2\}-\ell$ extra null singular values.}
With the aim of approximating the $r$-truncated SVD of $\Yck$, using the {notation  \eqref{eq:svd_part}}, we  write 
\[
    B_\ell = \underbrace{(U_{B})_1 (\Sigma_{B})_1 (V_{B})_1^T}_{=[B_{\ell}]_1} + \underbrace{(U_B)_{2} (\Sigma_B)_{2} (V_B)_{2}^T}_{=[B_{\ell}]_2}=[B_{\ell}]_1+[B_{\ell}]_2 ,
\]
with $(U_B)_1 \in \R^{(\ell+1) \times r}, (U_B)_2 \in \R^{(\ell+1) \times (\ell+1-r)}, (V_B)_1 \in \R^{\ell \times r}, (V_B)_2 \in \R^{\ell \times (\ell-r)}, \Sigma_1 \in \R^{r \times r}, \Sigma_2 \in \R^{(\ell+1-r) \times (\ell-r)}$. Then, \REV{using the notation $G_{\ell} = [G_{\ell}]_1+[G_{\ell}]_2$, where $[G_{\ell}]_1 = P_{\ell+1} [B_{\ell}]_1Q_\ell^T$ and $[G_{\ell}]_2 = P_{\ell+1} [B_{\ell}]_2Q_\ell^T$}, we can truncate the decomposition in \eqref{svd_Z} and define
\begin{equation*} \label{eq:wl}  
    \REV{[G_\ell]_1=} \underbrace{P_{\ell+1} (U_{B})_1}_{\tilde U_r} \underbrace{(\Sigma_{B})_1}_{\tilde \Sigma_r} \underbrace{(V_{B})_1^T Q_\ell^T}_{\tilde V_r^T}=\tilde U_r \tilde \Sigma_r \tilde V_r .
\end{equation*}

The preceding discussion indicates how singular values estimates can be obtained via the Lanczos method, but it reveals nothing about the approximation quality of the singular values and corresponding singular vectors. 
We refer to \cite[Section 10.4]{golub2013matrix} and references therein for the convergence results of the Lanczos procedure. Here, we highlight the following properties that are crucial for our analysis: 
\begin{itemize}
    \item in exact arithmetic the method  has finite termination; indeed, when $\ell =\min\{n_1,n_2\}$, the SVD in \eqref{svd_Z} is the exact SVD of $\Yck$. However, the procedure is generally implemented as an iterative procedure due to round-off errors and to avoid to perform a large number of steps;
    \item the inexact singular values $\tilde \sigma_{i,\ell}$ approximate the corresponding $\sigma_i$ from below \cite[Th.  6.4]{saad2011numerical};
    \item letting $(U_B)_{\ell +1, j}$ be the last element of the $j$-th column of $U_B$,  as shown in \cite{larsen1998lanczos}, we have that 
    \begin{equation} \label{eq:Lanczos_error_bound_v2}
        \min_{\sigma \in \sigma(\Yck)} | \sigma - \tilde \sigma_{j, \ell} | \le | \alpha_{\ell+1} | |(U_B)_{\ell +1, j}|.
    \end{equation}
\end{itemize}
In the classic implementations \cite{larsen1998lanczos}, the index $\ell$ is adaptively increased in order to reach the desired accuracy in the approximation of the first $r$ singular values, using the error bound given in \eqref{eq:Lanczos_error_bound_v2}, that is  based only on a byproduct of the Lanczos method.
\REV{In  the next subsection, we show  that also the accuracy requirements 
\eqref{eq:stopping_final} in iRAPM can be checked with no extra computational cost using information already provided by the Lanczos process.
}

\subsubsection{Inexactness Criteria}
Letting  $ \ML_\ell$ be the matrix  given in \eqref{svd_Z}, \ste{note that its first $\ell$ singular values are $\tilde \sigma_{1,\ell}, \dots, \tilde \sigma_{\ell,\ell}$, while the remaining are zero.} {Although the singular values $\tilde{\sigma}_{i,\ell}$, $i=1,\ldots,\ell$, are well-defined only for $\ell=1,\ldots,\min\{n_1,n_2\}$, we can artificially set $\tilde\sigma_{i,\ell} = \sigma_i$, $i=1,\ldots,\min\{n_1,n_2\}$ for $\ell > \min\{n_1, n_2\}$. We} assume that at step $\ell>r$ of the Lanczos procedure it holds $\tilde\sigma_{r,\ell} > \tilde\sigma_{r+1,\ell} $, i.e., the $r$ and $r+1$ singular values of $\cZ_\ell$ are distinct. This is a reasonable assumption since the method is expected to converge to a $r$-rank matrix. Then, let $\kappa_\ell\REV{>0}$ be given by 
\begin{equation}
    \label{eq:const_A}
    \kappa_\ell = \frac{2}{1-\gamma} \cdot \frac{(1-\gamma)\tilde\sigma_{r,\ell}+\gamma\tilde\sigma_{r+1,\ell}}{\tilde\sigma_{r,\ell}-\tilde\sigma_{r+1,\ell}},
\end{equation}
where $\gamma \in (0,1)$. For any $\ell>r$, we define the sequences
\begin{align}
    W^{\ell}&=
    \begin{cases}
        [G^{\ell}]_1, \quad & \text{if } \ell\leq \min\{n_1,n_2\},\\
        [Y_{reg}^k]_1, \quad & \text{if } \ell>\min\{n_1,n_2\}
    \end{cases}\label{eq:Wl}\\
    c^\ell &= {\sum_{i=r+1}^{\min\{\ell, n_1, n_2\}} \tilde \sigma_{i,\ell}^2}, \label{eq:cl} \\
    a^\ell &= 
    \begin{cases}
        \kappa_\ell \| \Yck - \cZ_\ell \| &\text{for } \ell \leq \min\{n_1,n_2\}, \\
        0 &\text{for } \ell > \min\{n_1, n_2\}
    \end{cases}. \label{eq:al}
\end{align}
In the following, we show that these sequences satisfy \eqref{eq:cond_1}-\eqref{eq:cond_3} and \eqref{eq:cond_4}.
\begin{lemma}
    The sequences \eqref{eq:Wl} and \eqref{eq:cl} satisfy conditions \eqref{eq:cond_1}-\eqref{eq:cond_3}.
\end{lemma}

\begin{proof}
From \cite[Theorem 6.4]{saad2011numerical}, we have that $\tilde \sigma_{i,\ell} \le \sigma_i$ for all $\ell<\min\{n_1,n_2\}$, and from the finite termination of the method we have $\tilde \sigma_{i,\ell} =\sigma_i$ for $\ell\geq \min\{n_1, n_2\}$. Then, 
\[
c_\ell = \sum_{i=r+1}^{\min\{\ell,n_1,n_2\}} \tilde \sigma_{i,\ell}^2 
       \le \sum_{i=r+1}^{\min\{\ell,n_1,n_2\}} \sigma_{i}^2 
       \le \sum_{i=r+1}^{\min\{n_1, n_2\}} \sigma_{i}^2 
       = \norm{\hat Y^{k+1} - \Yck}^2, \quad \forall \, \ell\in\mathbb{N},
\]
\[
    \lim_{\ell \rightarrow \infty} c_\ell = \lim_{\ell \rightarrow \infty} \sum_{i=r+1}^{\min\{\ell,n_1,n_2\}} \tilde \sigma_{i,\ell}^2     = \sum_{i=r+1}^{\min\{n_1, n_2\}} \sigma_{i}^2 = \norm{\hat Y^{k+1} - \Yck}^2.
\]
Then, \eqref{eq:cond_1} and \eqref{eq:cond_2} holds. {Condition \eqref{eq:cond_3} follows by the definition of $W^{\ell}$ in \eqref{eq:wl}.}
\end{proof}

\begin{lemma}
{Let $\tilde \sigma_{i,\ell}$, $i=1,\ldots,\ell$, denote the singular values of $\cZ_\ell$ and assume that $\tilde\sigma_{r,\ell}  > \tilde\sigma_{r+1,\ell}$.} Then, the sequences given {in \eqref{eq:wl}} and \eqref{eq:al} satisfy condition \eqref{eq:cond_4}.
\end{lemma}

\begin{proof}
{The limit in \eqref{eq:cond_4} holds by the finite termination of the Lanczos algorithm and the definitions in \eqref{eq:wl} and \eqref{eq:al}. We now show that the inequality in \eqref{eq:cond_4} holds. We limit ourselves to the case $\ell < \min\{n_1,n_2\}$, as otherwise the thesis holds by \eqref{eq:wl} and \eqref{eq:al}}. For simplicity, we omit the \ste{index $k$ and let $m=\min\{n_1,n_2\}$}. Using the notation given in \eqref{eq:svd_part}, we consider the block partitions
\begin{equation}
    Y_{reg} = \Tsvd{Y_{reg}} + \tsvd{Y_{reg}}, \quad \cZ_\ell      = {\Tsvd{\cZ_\ell}} + \tsvd{\cZ_\ell}.\label{Z_part}
\end{equation}
Then, we have that $ \hat Y^{k+1}=\Tsvd{Y_{reg}}$ and $ W^\ell=\Tsvd{\cZ_\ell}$.
The SVD of $Y_{reg}$ and $\cZ_\ell$ can be written as
\ste{
\begin{equation*}
    Y_{reg} = U_{Y_{reg}} \diag(\sigma_1, \dots, \sigma_{m}) V_{Y_{reg}}^T, \quad \cZ_\ell     = U_{\cZ_\ell}       \diag(\tilde\sigma_{1,\ell}, \dots,\tilde\sigma_{\ell,\ell}, 0,\ldots,0)        V_{\cZ_\ell}^T.
\end{equation*}}
Given $b_1\ge 0$ and $b_2\ge 0$, we observe that the SVD  of $\Tsvd{Y_{reg}} + b_1 \tsvd{Y_{reg}}$ and $\Tsvd{\cZ_\ell} + b_2 \tsvd{\cZ_\ell}$ are \ste{
\begin{align*}
    \Tsvd{Y_{reg}} + b_1 \tsvd{Y_{reg}} &= U_{Y_{reg}} \diag(\sigma_1, \dots, \sigma_r,b_1 \sigma_{r+1}, \dots, b_1 \sigma_m) V_{Y_{reg}}^T \\
    \Tsvd{\cZ_\ell} + b_2 \tsvd{\cZ_\ell}           &= U_{\cZ_\ell}        \diag(\tilde\sigma_{1,\ell}, \dots, \tilde\sigma_{r,\ell}, b_2 \tilde\sigma_{r+1,\ell}, \dots, b_2 \tilde\sigma_{\ell,\ell},0,\dots,0)           V_{\cZ_\ell}^T.
\end{align*}}
Then, we have:
\begin{align}
    \Tsvd{Y_{reg}} &= \proj_\setRank (\Tsvd{Y_{reg}} + b_1 \tsvd{Y_{reg}}), \quad \forall \ b_1 \in \ste{[0, \sigma_r/\sigma_{r+1}) }\label{eq:eckart1}\\
    \Tsvd{\cZ_\ell}      &= \proj_\setRank (\Tsvd{\cZ_\ell} + b_2 \tsvd{\cZ_\ell}),           \quad \forall\ste{ \ b_2 \in [0, \tilde\sigma_{r,\ell}/\tilde\sigma_{r+1,\ell})}.\label{eq:eckart2}
\end{align}
Equation \eqref{eq:eckart1} yields
$$
    \|\Tsvd{Y_{reg}}  - (\Tsvd{Y_{reg}}  + b_1 \tsvd{Y_{reg}})\|^2  \le \|\Tsvd{\cZ_\ell} -(\Tsvd{Y_{reg}}  + b_1 \tsvd{Y_{reg}}) \|^2.
$$
The above inequality implies
$$
    \|\Tsvd{\cZ_\ell} - \Tsvd{Y_{reg}} \|^2 + \|b_1 \tsvd{Y_{reg}}\|^2 - 2 \langle \Tsvd{\cZ_\ell} - \Tsvd{Y_{reg}}, b_1 \tsvd{Y_{reg}} \rangle \ge 
    \|b_1 \tsvd{Y_{reg}}\|^2,
$$
which yields     
\begin{equation}\label{eq:1}
    \frac{1}{2b_1} \|\Tsvd{\cZ_\ell} - \Tsvd{Y_{reg}}\|^2 -\langle \Tsvd{\cZ_\ell} - \Tsvd{Y_{reg}}, \tsvd{Y_{reg}} \rangle \ge 0.  
\end{equation}
\REV{Using  \eqref{eq:eckart2} and following a similar reasoning as before, we get}
\begin{equation}
     \frac{1}{2b_2} \|\Tsvd{Y_{reg}} - \Tsvd{\cZ_\ell} \|^2 - \langle \Tsvd{Y_{reg}} - \Tsvd{\cZ_\ell}, \tsvd{\cZ_\ell} \rangle\ge 0. \label{eq:2}
\end{equation}
By summing up \eqref{eq:1} and \eqref{eq:2}, we obtain:
\begin{equation}
     0 \le \left( \frac{1}{2b_1} + \frac{1}{2b_2} \right) \|\Tsvd{Y_{reg}} - \Tsvd{\cZ_\ell}\|^2 - \langle \Tsvd{\cZ_\ell} - \Tsvd{Y_{reg}}, \tsvd{Y_{reg}} - \tsvd{\cZ_\ell} \rangle. \label{eq:3}
\end{equation}
Recalling that $\tsvd{Y_{reg}} = Y_{reg}-\Tsvd{Y_{reg}}$ and $\tsvd{\cZ_\ell} = \cZ_\ell - \Tsvd{\cZ_\ell}$, we have:
$$
    -\left( \frac{1}{2b_1} + \frac{1}{2b_2} \right) \|\Tsvd{Y_{reg}} - \Tsvd{\cZ_\ell} \|^2 \le \langle \Tsvd{Y_{reg}} -\Tsvd{\cZ_\ell} , Y_{reg}-  \cZ_\ell \rangle -  \|\Tsvd{Y_{reg}} - \Tsvd{\cZ_\ell}\|^2, 
$$
which implies
$$
    \left( 1 - \frac{1}{2b_1} - \frac{1}{2b_2} \right) \|\Tsvd{Y_{reg}} - \Tsvd{\cZ_\ell}\|^2 \le \langle \Tsvd{Y_{reg}}-\Tsvd{\cZ_\ell} , Y_{reg} - \cZ_\ell \rangle.
$$
By limiting ourselves to the case where $\left( 1 - \frac{1}{2b_1} - \frac{1}{2b_2} \right) > 0$, which is verified by imposing $b_1, b_2 \ge 1$ (where equality cannot hold simultaneously for both), and using Cauchy-Schwarz inequality, we have:
\begin{align}
    \left( 1 - \frac{1}{2b_1} - \frac{1}{2b_2} \right) \|\Tsvd{Y_{reg}} - \Tsvd{\cZ_\ell}\| &\le \|Y_{reg} - \cZ_\ell\|. \label{eq:4}
\end{align}
Let us set $b_1 = 1$ and \ste{ $b_2 = \gamma + (1-\gamma) \tilde\sigma_{r,\ell}/\tilde\sigma_{r+1,\ell} = \frac{(1 -\gamma) \tilde\sigma_{r,\ell}+\gamma \tilde\sigma_{r+1,\ell}}{\tilde\sigma_{r+1,\ell}}$}, with $\gamma \in (0,1)$. With this choice of $b_1$ and $b_2$, the constant in $\eqref{eq:4}$ becomes:
\begin{align*}
    \left( 1 - \frac{1}{2b_1} - \frac{1}{2b_2} \right) &= \left( 1 - \frac{1}{2} - \ste{\frac{\tilde\sigma_{r+1,\ell}}{2((1-\gamma)\tilde\sigma_{r,\ell}+ \gamma \tilde\sigma_{r+1,\ell})}} \right) \\
    &= \frac{1}{2} \left( 1 - \ste{\frac{\tilde\sigma_{r+1,\ell}}{(1-\gamma)\tilde\sigma_{r,\ell}+ \gamma \tilde\sigma_{r+1,\ell}}} \right)
    = \frac{(1-\gamma)}{2} \cdot\ste{\frac{\tilde \sigma_{r,\ell}-\tilde\sigma_{r+1,\ell}}{(1-\gamma)\tilde\sigma_{r,\ell}+\gamma\tilde\sigma_{r+1,\ell}}} = \frac{1}{\kappa_\ell},
\end{align*}
where $\kappa_\ell$ is given in \eqref{eq:const_A}.
Then,  from \eqref{eq:4}, we have:
\[
    \REV{\|\Tsvd{Y_{reg}} - \Tsvd{\cZ_\ell}\| \leq  \kappa_{\ell} \|Y_{reg} - \cZ_\ell\| = a^\ell,}
\]
and the first part of \eqref{eq:cond_4} is verified.
\end{proof}

\subsubsection{Computation of the right-hand side of the stopping criteria \eqref{eq:stopping_final}}

The criteria \eqref{eq:stopping_final} require to compute additional quantities, in particular the Frobenius norms $\| W^\ell - \Yck \|$ and $\| \cZ_\ell - \Yck \|$. We show that these two quantities can be computed without additional cost as a byproduct of the Lanczos process.

We start recalling few useful relations. By \eqref{eq:Lanczos_recurrence_matrix_1} it follows $P_{\ell+1}^T Y_{reg}^k Q_{\ell} = B_{\ell}$ and by \eqref{svd_Z} we have $\cZ_\ell = P_{\ell+1} B_\ell Q_\ell^T$. Moreover, since the columns of $P_{\ell + 1}$ and $Q_{\ell}$ are orthonormal, it holds:
\begin{align}
	p_{i}^T p_{j} &= 0, \quad 1 \le i,j \le \ell+1, \ i \neq j, \label{eq:orth_p} \\
	q_{i}^T q_{j} &= 0, \quad 1 \le i,j \le \ell, \ i \neq j, \label{eq:orth_q} \\
	(I - P_{\ell+1} P_{\ell+1}^T)p_{i}&=0,
    \quad 1 \le i \le \ell+1. 
    \label{eq:orth_I_PP}
\end{align}
Inspired by \cite{Simon2000}, in the next Lemma we show that $\| \Yck-\cZ_\ell \| ^2$ can be computed by recurrence exploiting the entries of $B_{\ell}$. 

\begin{lemma} \label{lem:criterion1}
	Let us denote $\omega_\ell = \| \Yck-\cZ_\ell \| ^2$ and let $\alpha_\ell, \ \beta_{\ell+1}$ be given in \eqref{eq:Lanczos_recurrence_vector_1}-\eqref{eq:Lanczos_recurrence_vector_2}. Then,
	\[
	\omega_{\ell-1} = \omega_{\ell} + \alpha_{\ell}^2 + \beta_{\ell+1}^2.
	\]
\end{lemma}

\begin{proof}
For the sake of simplicity, we omit the index $k$. From $\eqref{eq:Lanczos_recurrence_matrix_2}$ we have:
\begin{align}
    Y_{reg} - G_\ell = Y_{reg} - P_{\ell+1} B_\ell Q_\ell^T &= Y_{reg} - P_{\ell+1} P_{\ell+1}^T Y_{reg}+ \alpha_{\ell+1} P_{\ell+1} e_{\ell+1} q_{\ell+1}^T \nonumber\\
    &= (I - P_{\ell+1} P_{\ell+1}^T) Y_{reg} + \alpha_{\ell+1} p_{\ell+1} q_{\ell+1}^T. \label{eq:A-Z}
\end{align}
We observe that:
\begin{equation}
    \label{eq:1:bis}
    (I - P_{\ell+1} P_{\ell+1}^T) Y_{reg} = (I - P_{\ell} P_{\ell}^T) Y_{reg} - p_{\ell+1} p_{\ell+1}^T Y_{reg}.
\end{equation}
Then, \eqref{eq:A-Z}, \eqref{eq:1:bis} and \eqref{eq:Lanczos_recurrence_vector_1}, yield:
\begin{align*}
    \omega_{\ell-1} &= \| Y_{reg}-G_{\ell-1} \| ^2 = \| (I - P_{\ell} P_{\ell}^T) Y_{reg} +\alpha_{\ell} p_{\ell} q_{\ell}^T \| ^2 \\
    &= \| (I - P_{\ell+1} P_{\ell+1}^T) Y_{reg} + {p_{\ell+1} p_{\ell+1}^T Y_{reg}} + \alpha_{\ell} p_{\ell} q_{\ell}^T \| ^2 \\
    &= \| (I - P_{\ell+1} P_{\ell+1}^T) Y_{reg} + {\alpha_{\ell+1} p_{\ell+1} q_{\ell+1}^T + \beta_{\ell+1} p_{\ell+1} q_{\ell}^T} + \alpha_{\ell} p_{\ell} q_{\ell}^T \| ^2 \\
    &= {\omega_{\ell}}  + \| \beta_{\ell+1} p_{\ell+1} q_{\ell}^T + \alpha_{\ell} p_{\ell} q_{\ell}^T \| ^2 + 2 \langle (I - P_{\ell+1} P_{\ell+1}^T) Y_{reg} + \alpha_{\ell+1} p_{\ell+1} q_{\ell+1}^T, \beta_{\ell+1} p_{\ell+1} q_{\ell}^T + \alpha_{\ell} p_{\ell} q_{\ell}^T \rangle.
\end{align*}
By using \eqref{eq:orth_p}, \eqref{eq:orth_q}, and \eqref{eq:orth_I_PP}, it follows that the Frobenius inner product is null. In addition, both $p_{\ell+1} q_{\ell}^T$ and $p_{\ell} q_{\ell}^T$ are rank-one matrices and the vectors  $p_{\ell+1},p_{\ell}$ and $q_{\ell}$ are unitary vectors, then $\|p_{\ell+1} q_{\ell}^T\| =1$, $\|p_{\ell} q_{\ell}^T\| =1$ and $\| \beta_{\ell+1} p_{\ell+1} q_{\ell}^T + \alpha_{\ell} p_{\ell} q_{\ell}^T \| ^2 = \alpha_{\ell}^2 + \beta_{\ell+1}^2$ and the thesis follows.
\end{proof}

The following Lemma shows the relationship between $\| \Yck - W^\ell \| ^2$ and $\|\Yck-\cZ_\ell \| ^2$.
\begin{lemma}\label{lem:criterion2}
    Let $\tilde \sigma_{i,\ell}$, $i=1,\ldots,\ell$, denote the singular values of $\cZ_\ell$. Then we have
	\[
        \| \Yck - W^\ell \| ^2 = \| \Yck-\cZ_\ell \| ^2 + \sum_{i=r+1}^\ell \tilde \sigma_{i,\ell}^2.
	\]
\end{lemma}

\begin{proof}
We omit the iteration index $k$, for the sake of simplicity. We recall that \REV{$W^\ell=[\cZ_\ell]_1$, in addition \eqref{Z_part} holds and $[\cZ_\ell]_{2}=P_{\ell+1} [B_\ell]_{2} Q_\ell^T$, where  $[B_\ell]_2=(U_B)_{2} (\Sigma_B)_{2} (V_B)_{2}$.
Then:
\begin{align*}
    \| Y_{reg}-W^\ell \| ^2& = \| Y_{reg}- [\cZ_\ell]_1 \| ^2 
    = \| Y_{reg}- \cZ_\ell+[\cZ_\ell]_{2} \| ^2\\
    &\ste{= \| Y_{reg}- \cZ_\ell \|^2 + \| [\cZ_\ell]_{2} \| ^2 + 2 \langle Y_{reg}- \cZ_\ell, P_{\ell+1} [B_\ell]_{2} Q_\ell^T \rangle} .
\end{align*}}
\REV{We can show that the Frobenius inner product is null.} Indeed, by rewriting $Y_{reg}- \cZ_\ell$ as in \eqref{eq:A-Z}, we have:
\begin{align}
    \langle Y_{reg}- \cZ_\ell, P_{\ell+1} [B_\ell]_{2} Q_\ell^T \rangle  &= 
    \langle (I - P_{\ell+1} P_{\ell+1}^T) Y_{reg}  + \alpha_{\ell+1} p_{\ell+1} q_{\ell+1}^T, P_{\ell+1} [B_\ell]_{2} Q_\ell^T \rangle  \nonumber\\
    & = \langle (I - P_{\ell+1} P_{\ell+1}^T) Y_{reg}, P_{\ell+1} [B_\ell]_{2} Q_\ell^T \rangle + \alpha_{\ell+1}\langle p_{\ell+1} q_{\ell+1}^T, P_{\ell+1} [B_\ell]_{2} Q_\ell^T \rangle  \nonumber\\
    &= \textrm{tr} \ste{\left( Y_{reg}^T (I - P_{\ell+1} P_{\ell+1}^T) P_{\ell+1} [B_\ell]_{2} Q_\ell^T \right) + \alpha_{\ell+1} \textrm{tr} \left( q_{\ell+1} p_{\ell+1}^T P_{\ell+1} [B_\ell]_{2} Q_\ell^T \right)}. \label{eq:scaler_prod_trace}
\end{align}
The first term in \eqref{eq:scaler_prod_trace} is null by \eqref{eq:orth_I_PP}, while the second, for the properties of the trace, becomes $\textrm{tr} \left( Q_\ell^T q_{\ell+1} p_{\ell+1}^T P_{\ell+1} [B_\ell]_{2} \right)$ and then it is null for \eqref{eq:orth_q}.
On the other hand, by the definition of Frobenius norm, we have $\| \tsvd{\cZ_\ell} \| ^2 = \sum_{i=r+1}^\ell \tilde \sigma_{i,\ell}^2$. Then, the thesis follows.
\end{proof}

%%%%%%%%%%%%%%%%%%%%%%%%%%
% Numerical Illustration %
%%%%%%%%%%%%%%%%%%%%%%%%%%
\section{Numerical illustration}

In this section, we  provide numerical tests which illustrate the behaviour of the iRAPM procedure. We focused on the Matrix Completion (MC) problem, which consists in recovering a matrix $M\in\R^{n_1\times n_2}$ from the knowledge of a subset of its entries; namely, the entries with indexes in the index set $\Omega = \{(i,j): \ M_{i,j} \text{ is known}\}$. In the following, we will write $(M)_\Omega\in\R^{n_1\times n_2}$ to denote the matrix obtained by keeping the entries of $M$ with indices in $\Omega$ and setting to 0 the others, and analogously for $(M)_{\bar \Omega}\in\R^{n_1\times n_2}$ with $\bar \Omega$ the complement of $\Omega$. Without further constraints, there exists an infinite number of matrices $Z\in\mathbb{R}^{n_1\times n_2}$ that comply with the condition $Z_{i,j} = M_{i,j}$ for all $(i,j)\in\Omega$. On the other hand, in most applications, one can assume that the matrix to be recovered has low rank, for more details, see \cite{candes2012exact}. Problems that can be approached with MC arise in various fields, such as recommendation systems \cite{ramlatchan2018survey}, image reconstruction \cite{nguyen2019low}, and data imputation \cite{agarwal2018model, Silei_energies}.

Then, the MC problem is generally formulated as the affine rank minimization problem
\begin{equation}
	\label{eq:min_rank}
	\begin{aligned}
		\min \quad & \textrm{rank}(Z) \\    
		\textrm{s.t.} \quad & (Z)_\Omega = (M)_\Omega
	\end{aligned} .
\end{equation}
This is a combinatorial optimization problem that requires a worst-case exponential running time.

In the literature, problem \eqref{eq:min_rank} has been approached in different ways, for example, replacing the rank function with a convex function like the nuclear norm \cite{candes2012exact}, that is, the sum of the singular values of the matrix, or reformulating it as Semidefinite Programming Problem \cite{Bellavia21}. 

We are interested in reformulating the MC problem as the feasibility problem \eqref{eq:feasibility}. In fact, given an estimate $r$ of the rank of the sought matrix, we can reformulate it as the problem of finding a matrix in the intersection of the {rank level set} and the set of matrices that match the observed entries, i.e.:
\begin{equation}\label{eq:completion}
	\text{find }Z\in\R^{n_1\times n_2}\text{ such that }Z\in \setNorm\cap \setRank, 
\end{equation}
where
\begin{align}\label{eq:constraints_MC}
	\setNorm &= \{Z\in \mathbb{R}^{n_1\times n_2}: \ Z_{i,j} = M_{i,j}, \ \forall \ (i,j)\in\Omega\},\\
	\setRank &= \{Z\in\mathbb{R}^{n_1\times n_2}: \rank{Z}\leq r\}.
\end{align}
As the function that associates $M$ with $(M)_\Omega$ is a linear map, $\mathcal{C} $ is an affine set and the projection onto it is given by
\begin{equation}
	\projNorm (Z) = (M)_\Omega + (Z)_{\bar\Omega}.
\end{equation}

{The inexact projection onto $\setRank$ in iRAPM is computed  via the Lanczos method described in Section~\ref{sec:Lanczos}, equipped with the stopping criteria \eqref{eq:stopping_final}, while in APM and RAPM we use the standard stopping criterion based on the bound \eqref{eq:Lanczos_error_bound_v2}. We note that also in APM and RAPM the projection is not exact; however, we expect it to be more accurate than in iRAPM, more expensive and not justified to preserve the overall convergence behaviour.}

{For all the three methods under investigation, our implementation of the Lanczos procedure is based on the PROPACK code by Larsen \cite{larsen1998lanczos}, where we changed the stopping criterion when used in iRAPM. In APM and RAPM we used the standard stopping criterion and default tolerances, that is, $16 \varepsilon_m$, where $\varepsilon_m$ denotes the machine precision}. 

\subsection{\REV{Projection} operators: computational costs}

Before presenting the numerical results, we analyze the computational costs of \REV{the projection operators involved in the MC problem \eqref{eq:completion}.}

As observed above, \REV{projecting} onto $\setNorm$ consists in an element-wise replacement and can be computed in $n_1n_2$ operations. On the other hand, the projection onto $\setRank$ via the Lanczos method described in Section~\ref{sec:Lanczos} has two main parts that concur to the final cost: updating the Lanczos vectors \eqref{eq:Lanczos_recurrence_vector_1}-\eqref{eq:Lanczos_recurrence_vector_2}, and checking the criterion \eqref{eq:Lanczos_error_bound_v2} for APM and RAPM (the criteria \eqref{eq:stopping_final} for iRAPM, respectively).

The Lanczos procedure requires at each iteration two matrix-vector products, that is, $O(n_1n_2)$ operations. Then, assuming that the final Krylov subspace dimension at the $k$-th outer iteration is $\bar \ell_k$, the final cost for building the Krylov subspace is of the order $O(n_1n_2\bar\ell_k)$.

The standard criterion  \eqref{eq:Lanczos_error_bound_v2}, employed in APM and RAPM, requires to check $r$ inequalities at most, then it costs $O(r)$ each time that the stopping criterion is checked. The evaluation of the  two Frobenius norms required in the criteria \eqref{eq:stopping_final} in the iRAPM procedure can be significantly reduced exploiting Lemma~\ref{lem:criterion1} and~\ref{lem:criterion2}. In fact, the two norms can be computed as follows:
\ste{
\begin{equation}
 \label{eq:closed_form_norm_Z_Y}
\begin{array}{l}
	\| \Yck - \cZ_\ell \| ^2 = \| \Yck {-\cZ_0}\| ^2 - \sum_{i = 1}^{\ell} \alpha_i^2 - \sum_{j = 2}^{\ell+1} \beta_j^2, \\  \| \Yck - W^\ell \| ^2 = \| \Yck-\cZ_\ell \| ^2 + \sum_{i=r+1}^\ell \tilde{\sigma}_{i,\ell}^2.
\end{array}
\end{equation}}
where $\cZ_\ell$ is given in \eqref{svd_Z}. Then, using \eqref{eq:closed_form_norm_Z_Y}, the inexact criteria require to compute the Frobenius norm $\| \Yck - \cZ_0\| $, whose cost is $O(n_1n_2)$, only once at the beginning of the Lanczos procedure, \REV{$3\bar\ell_k-r$ sums} and the check of only 2 inequalities. 
Then, in all the alternating projection approaches under investigation, the cost of checking the stopping criteria is dominated by the cost of the Lanczos iterations, and we can conclude that, at a generic iteration $k$, the projection onto $\setRank$ costs $O(n_1n_2\bar\ell_k)$.

We stress that the strategy for increasing the dimension of the Krylov subspace has a crucial impact on the final cost of the projection onto $\setRank$. {The PROPACK code implements an adaptive strategy to reduce the number of criteria checks.} This strategy can result in selecting unnecessarily large values of $\bar\ell_k$. In our implementation of the three alternating projection methods, we chose to check the stopping criteria at each Lanczos iteration {in order to select the minimum Krylov dimension for which the criteria are satisfied.}

\subsection{Experiments on Matrix Completion}

For our experiments, we consider two different types of matrices: random Gaussian matrices of rank $r$ obtained as $M = LR^T \in \R^{n_1 \times n_2}$, where $L \in \R^{n_1 \times r}$, $R \in \R^{n_2 \times r}$ are matrices whose entries are sampled from a normal distribution, and the $r$-truncated SVD of the  {grayscale}  $512 \times 512$ {\tt lake} image, see  
Figure~\ref{fig:lake_and_val_sing} (left). 
We choose $n_1=n_2=512$ for the Gaussian matrices, and in both cases  we choose $r=30$. Then, we built the Gaussian matrices and truncated the image accordingly.

{The Gaussian matrices and the truncated image have different  singular value distributions as illustrated in  Figure~\ref{fig:lake_and_val_sing} (right), where we clearly see that the singular values of the truncated image decade faster than those of the Gaussian matrix. 
Then, this choice of the matrices allows us to  test  iRAPM against  different   singular value distributions.}

\begin{figure}
	\centering
	\includegraphics[width=0.4\textwidth]{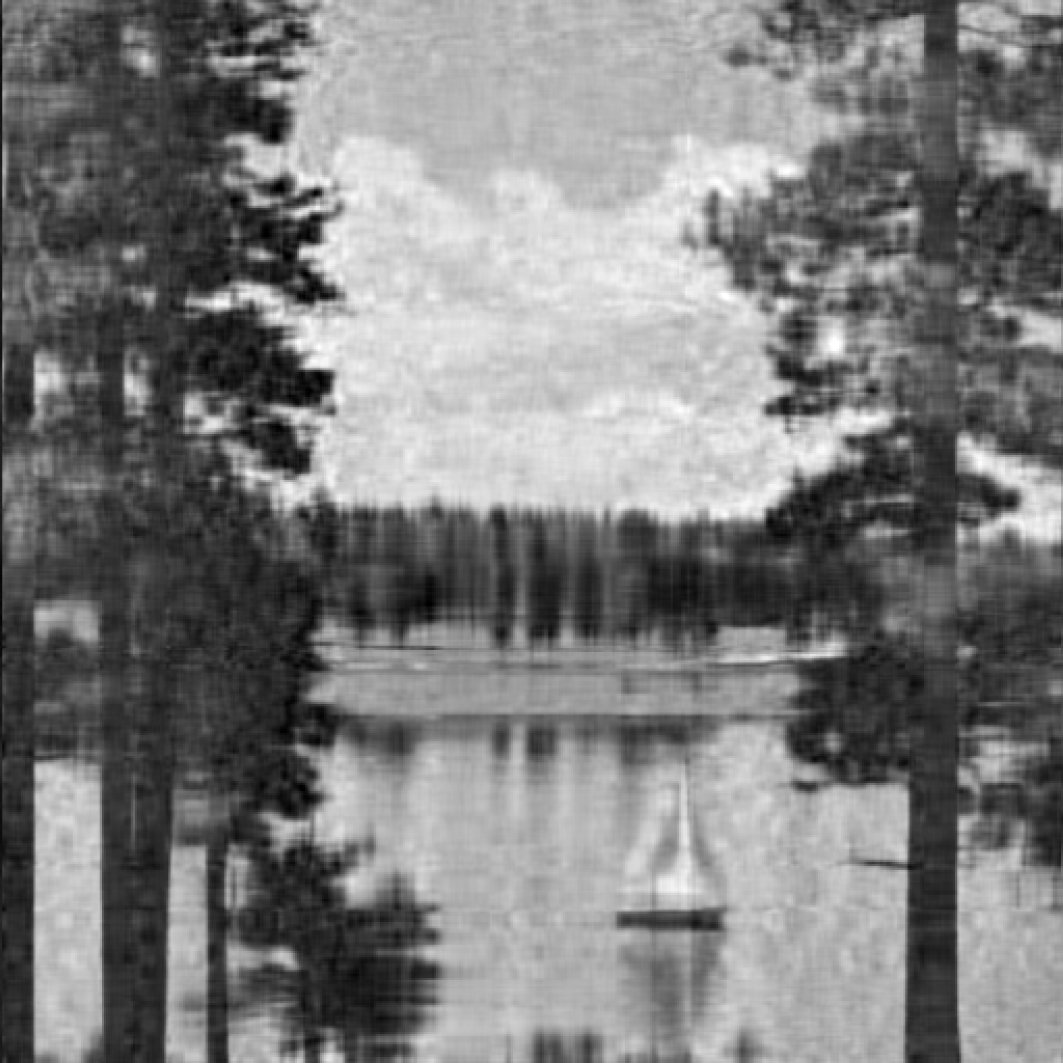}
	\includegraphics[width=0.4\textwidth]{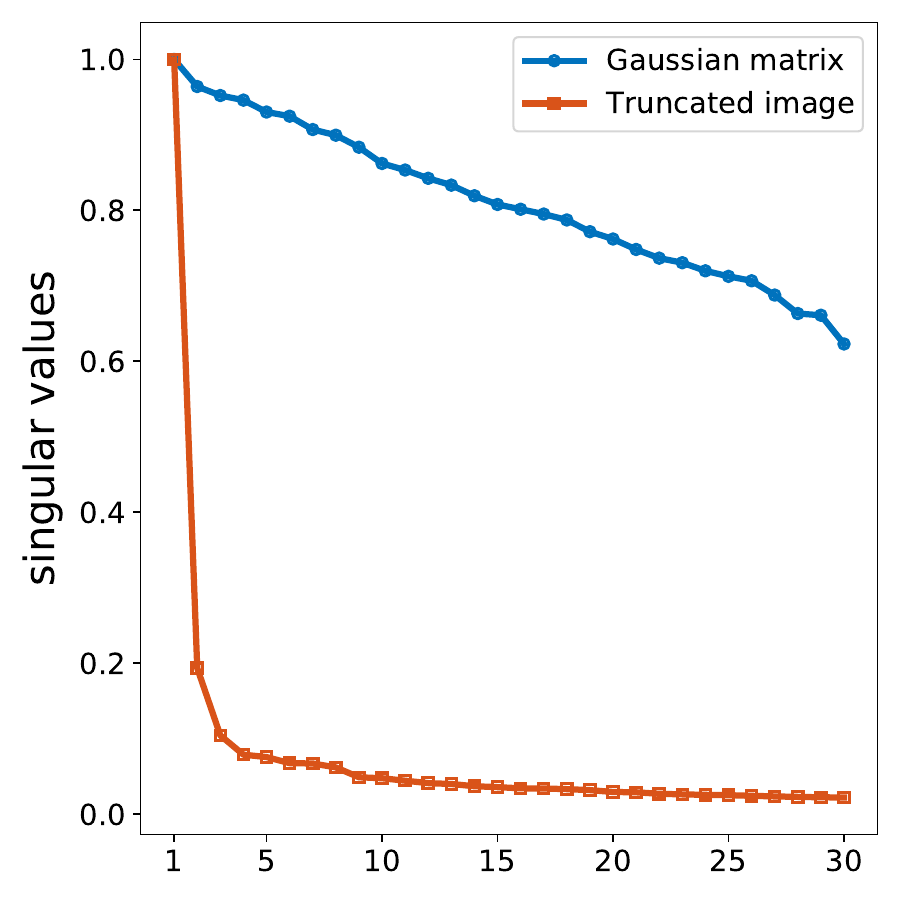}
	\caption{Truncated {\tt lake} image (left) and normalized singular values  of a Gaussian matrix (blue circles) and of the truncated image (red squares) (right).}
	\label{fig:lake_and_val_sing}
\end{figure}
 
\REV{The set $\Omega$ is built sampling $q$ entries uniformly at random, where the number of observed entries $q$ is chosen such that the ratio $q/((n_1+n_2-r)r) = 2.6$ in both cases.}
We consider the choice $X^0 = (M)_\Omega, \ Y^0 = \projRank (X^0)$ as starting points for RAPM and iRAPM, while APM only needs  $Y^0$. 
We compare the performance of the methods by performing  a prefixed number of iterations $k_{f}=200$. In this way, we can observe the effective saving of inexact projections in iRAPM along iterations while comparing  the final reconstruction error. 

Two metrics are considered to monitor the performance of the methods:
\begin{align}
	e_{\Omega}(Y^k)  &= \REV{\frac{\| (M)_\Omega - (Y^k)_\Omega \|}{\| (M)_\Omega \|}, \label{eq:metric_ome} }\\
    e_{mse}(Y^{k_f}) &= \frac{\| M - Y^{k_{f}} \|^2}{n_1 n_2}. \label{eq:metric_mse} 
\end{align}
The former measures the error on the observed entries along the iterations, while the latter is the mean squared error on all the entries at termination and  it measures the quality of the final reconstruction.

{As a measure of the overall computational cost needed to reach iteration $k$ and to compare the three approaches under investigation,  we consider the quantity 
$$
    \mbox{\tt cost}(k)=\sum_{i=1}^k (\bar \ell_{\REV{i}}-r),
$$
as at each iteration $k$, the Lanczos process is started with a Krylov subspace of dimension $r$, then we can obtain computational saving only reducing $\bar \ell_k-r$.
}
In Tables~\ref{tab:Gauss_table} and ~\ref{tab:Lake_table} we provide statistics  of the three considered methods; \REV{for the regularized methods we used $\lambda_k = \mu_k = 16$ and} iRAPM was tested  with different values of   $\zeta$ in \eqref{eq:stopping_final}.
The reported statistics are the means over 5 random runs. 
\REV{Columns 3 and 4 show the values of the two metrics \eqref{eq:metric_ome} and \eqref{eq:metric_mse} at termination, and we also report in brackets the  standard deviation of the mean-squared error. Finally, column 5, reports the quantity $\mbox{\tt cost}(k_f)$.}
As a first comment, we observe that the choice of $\zeta$ in iRAPM does not seem to be critical, and any value smaller than $10^{-3}$ provides an overall reduction of the computational cost as the dimension of the Krylov subspaces employed in iRAPM is generally smaller than those employed in RAPM.

{In Figures~\ref{fig:eomega}-\ref{fig:krylov} we report information on the behavior of the procedures along the iterations; iRAPM was run using $\zeta = 10^{-7}$. In Figure~\ref{fig:eomega}, we plot the average over all the runs of the metric \eqref{eq:metric_ome} versus the iterations (line), while the colored bands represent the minimum-maximum range of the metric. In the remaining two figures we focus on a representative single run: in Figure ~\ref{fig:cost} we plot for each method the ratio between $\mbox{\tt cost}(k)$ and the corresponding quantity for APM.  
This way, we compare the effective computational costs of RAPM and iRAPM as a fraction of the cost of APM. The plots in Figure~\ref{fig:krylov} show two different quantities: the solid lines represent $\bar \ell_k$ versus the iterations, while the dashed lines display the number of singular values that satisfy the criterion \eqref{eq:Lanczos_error_bound_v2}, i.e. the number of singular values computed by iRAPM that are approximated with the standard accuracy requirement.}
 
{Focusing on the Gaussian matrices, we observe that  RAPM and iRAPM are comparable in terms of accuracy and only slightly worse than APM. This is outlined in the statistics in Table~\ref{tab:Gauss_table} and in Figure~\ref{fig:eomega}. Figure~\ref{fig:cost} clearly  shows that RAPM and APM require approximately the same computational effort, while the saving provided by the use of the new accuracy requirements \eqref{eq:stopping_final} in iRAPM is relevant in the first phase of the convergence history (the first 20 iterations), while from the 20th iteration onward the savings settle around the $15\%$.
This is also evident in Figure~\ref{fig:krylov}  where we can see that iRAPM employes  Krylov subspaces of dimension  lower than 40 in the first 20 outer iterations, while higher dimensions are needed by RAPM and APM.  
On the other hand, we observe that  the Lanczos procedure  provides accurate singular values in the final phase of the convergence history employing Krylov subspaces of dimension slightly larger than $r$;  this behaviour can be ascribed to the fact that $\Yck$ becomes closer and closer to a rank $r$ matrix, and ideally, in case of rank $r$ matrices, the Lanczos process stops with $\ell$ no greater than $r+1$ \cite{saad2011numerical}.
It is also interesting to observe that the approximated singular values used to compute the projection in iRAPM are, as expected,  less accurate than in RAPM and APM, except for approximately 50 consecutive outer iterations starting from the 100th, and few more isolated ones. However, the overall convergence behavior and the final accuracy are comparable to those of APM and RAPM, showing that the standard criterion, except when the iterates are close to the sought solution, yields to the  
{\it oversolving} phenomena, i.e. the singular values are approximated with an accuracy higher than those needed for the convergence of the overall procedure.}

{Concerning  the \verb*|lake| image, as before, the performances of the three methods are comparable in terms of accuracy, but this time iRAPM achieves an error similar to APM outperforming the RAPM variant.
Figure~\ref{fig:cost} (right) shows that in this case the saving of the inexact approach in term of computational cost  is even more relevant than in the case of Gaussian matrices. Indeed, in the first iterations, saving is over the $50\%$, and later on the savings are approximately the $40\%$. From Figure~\ref{fig:krylov}, we can observe that the number of singular values satisfying \eqref{eq:Lanczos_error_bound_v2} is less than $r$ along the whole process. This confirms the {\it oversolving} phenomena noticed for the Gaussian matrices.

The experiments offer different insights on the behavior of alternating projection approaches and of iRAPM.
First, all the considered methods  manage to approximate the sought matrix with good accuracy and the error on the known  entries steadily decreases along the iterations. 
Second, it is evident the reduction in the overall computational cost provided by iRAPM.
Remarkably, this computational saving  is also obtained  on the \REV{realistic case of the} \verb*|lake| image.}

\begin{table}
    \small
	\centering
	\begin{tabular}{ccccc}
		\toprule
		{\tt Method} & $\zeta$ & $e_{\Omega}$ & $e_{mse}$ & $\mbox{\tt cost}(k_f)$ \\
		\midrule
		APM          &       & 2.968e-06 & 3.309e-09 (5.609e-09) & 836 \\
		RAPM         &       & 7.638e-06 & 1.733e-08 (2.378e-08) & 870 \\
		iRAPM        & 1e-09 & 8.229e-06 & 2.101e-08 (3.046e-08) & 715 \\
		iRAPM        & 1e-07 & 8.105e-06 & 1.889e-08 (2.462e-08) & 713 \\
		iRAPM        & 1e-05 & 8.204e-06 & 2.054e-08 (2.925e-08) & 712 \\
		iRAPM        & 1e-03 & 7.606e-06 & 1.699e-08 (2.307e-08) & 728 \\
		% iRAPM        & 1e-02 & 7.720e-06 & 1.796e-08 (2.510e-08) & 10212 \\
		\bottomrule
	\end{tabular}
	\caption{Gaussian matrices: statistics of the runs, $maxit=200$.}
	\label{tab:Gauss_table}
\end{table}

\begin{table}
    \small
	\centering
	\begin{tabular}{ccccc}
		\toprule
		{\tt Method} & $\zeta$ & $e_{\Omega}$ & $e_{mse}$ & $\mbox{\tt cost}(k_f)$\\
		\midrule
		APM          &       & 5.009e-04 & 1.606e-06 (8.758e-07) & 1654 \\
		RAPM         &       & 7.785e-04 & 3.092e-06 (1.374e-06) & 1722 \\
		iRAPM        & 1e-09 & 5.633e-04 & 2.061e-06 (1.118e-06) & 956 \\
		iRAPM        & 1e-07 & 5.181e-04 & 1.733e-06 (1.089e-06) & 969 \\
		iRAPM        & 1e-05 & 5.056e-04 & 1.609e-06 (1.388e-06) & 936 \\
		iRAPM        & 1e-03 & 7.652e-04 & 2.934e-06 (1.520e-06) & 1244 \\
		% iRAPM        & 1e-02 & 7.603e-04 & 3.080e-06 (1.451e-06) & 26151 \\
		\bottomrule
	\end{tabular}
	\caption{ {\tt lake} image: statistics of the runs, $maxit=200$}
	\label{tab:Lake_table}
\end{table}

\begin{figure}
    \centering
    \includegraphics[width=0.4\textwidth]{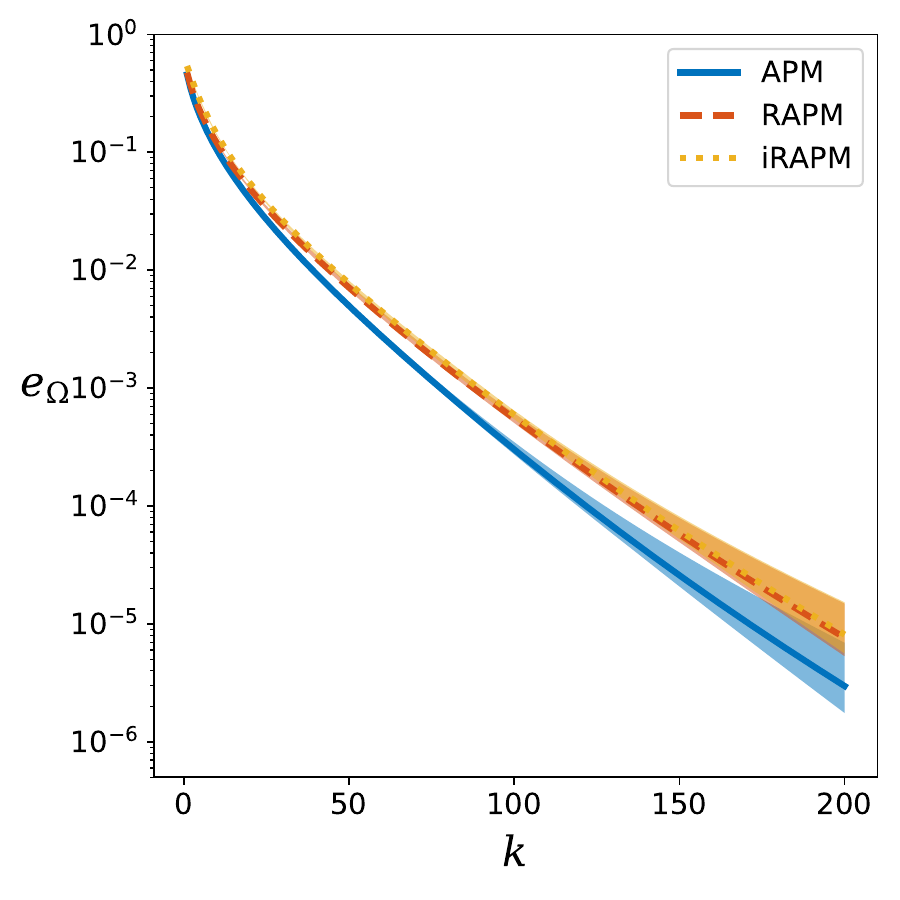}
    \includegraphics[width=0.4\textwidth]{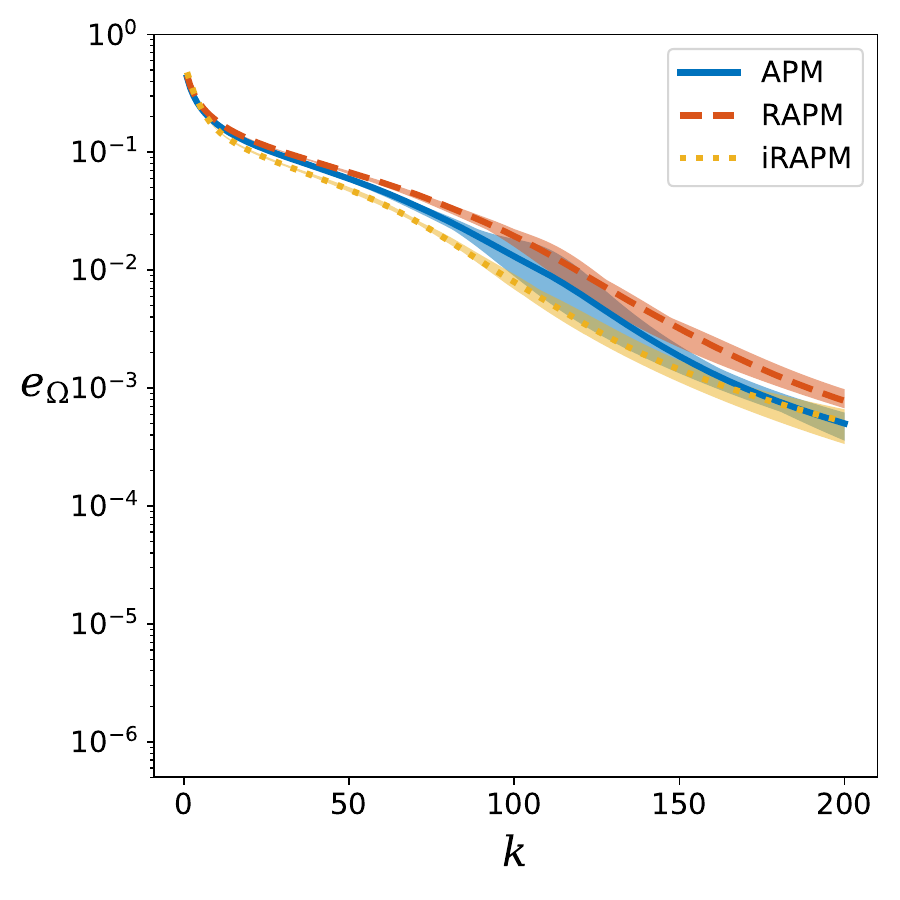}
    \vspace{-\baselineskip}
    \caption{Mean value and min-max band (over 5 runs) of $e_\Omega$ along the iterations. Gaussian matrices (left) and truncated {\tt lake} (right).}
    \label{fig:eomega}
\end{figure}

\begin{figure}
    \centering
    \includegraphics[width=0.4\textwidth]{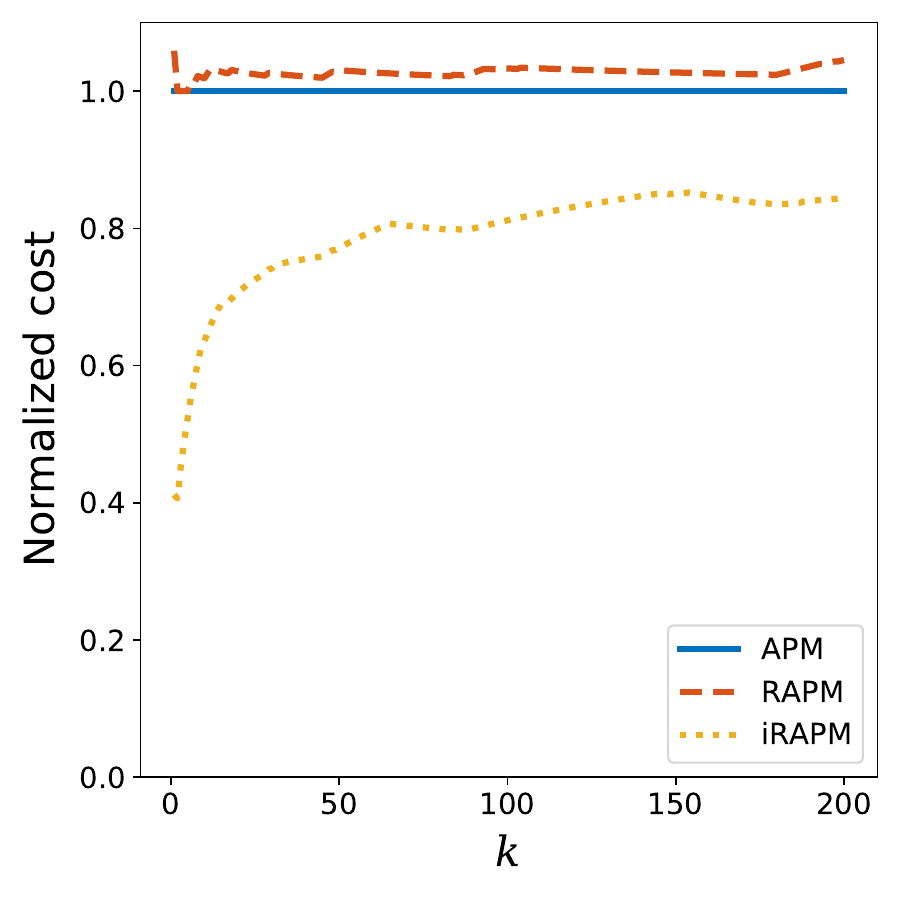}
    \includegraphics[width=0.4\textwidth]{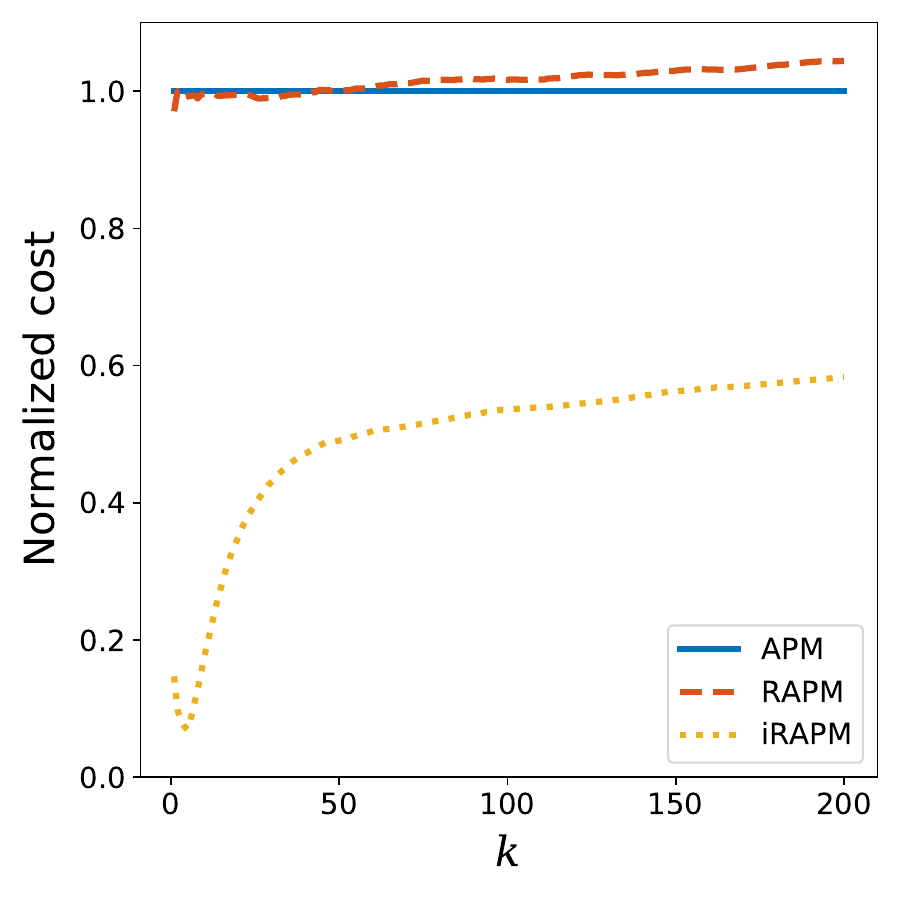}
    \vspace{-\baselineskip}
    \caption{For each method,  {\tt cost}$(k)$ over the corresponding  quantity for APM. Gaussian matrices (left) and truncated {\tt lake} (right).}
    \label{fig:cost}
\end{figure}

\begin{figure}
    \centering
    \includegraphics[width=0.4\textwidth]{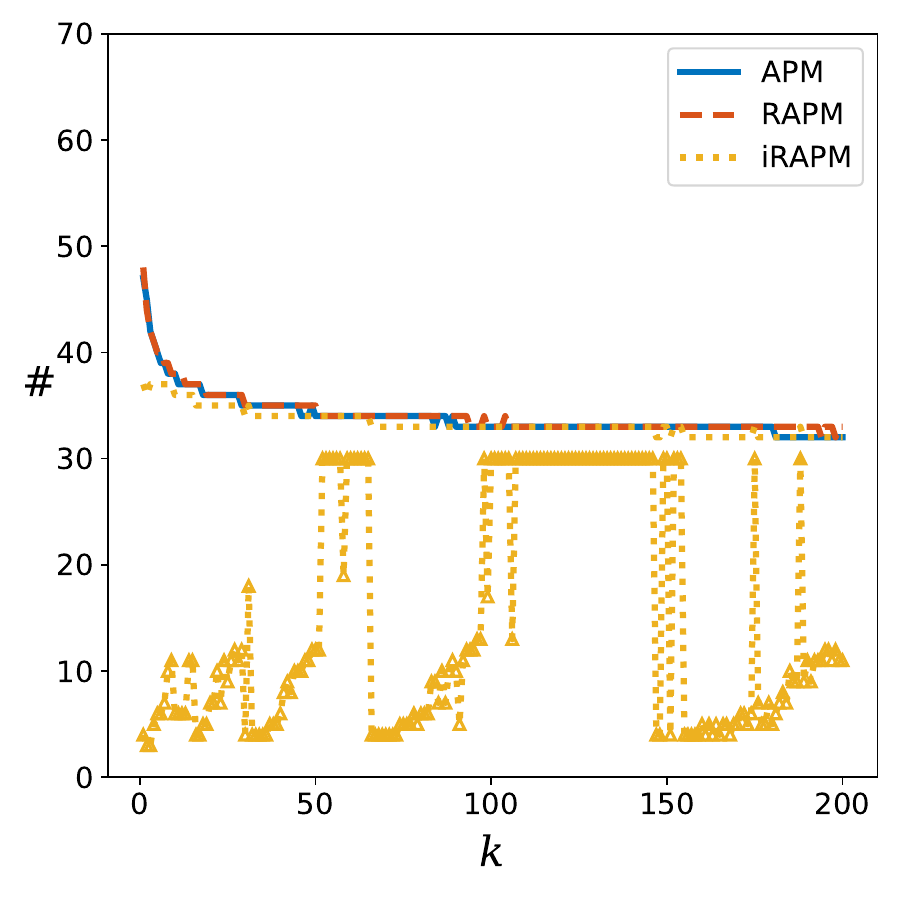}
    \includegraphics[width=0.4\textwidth]{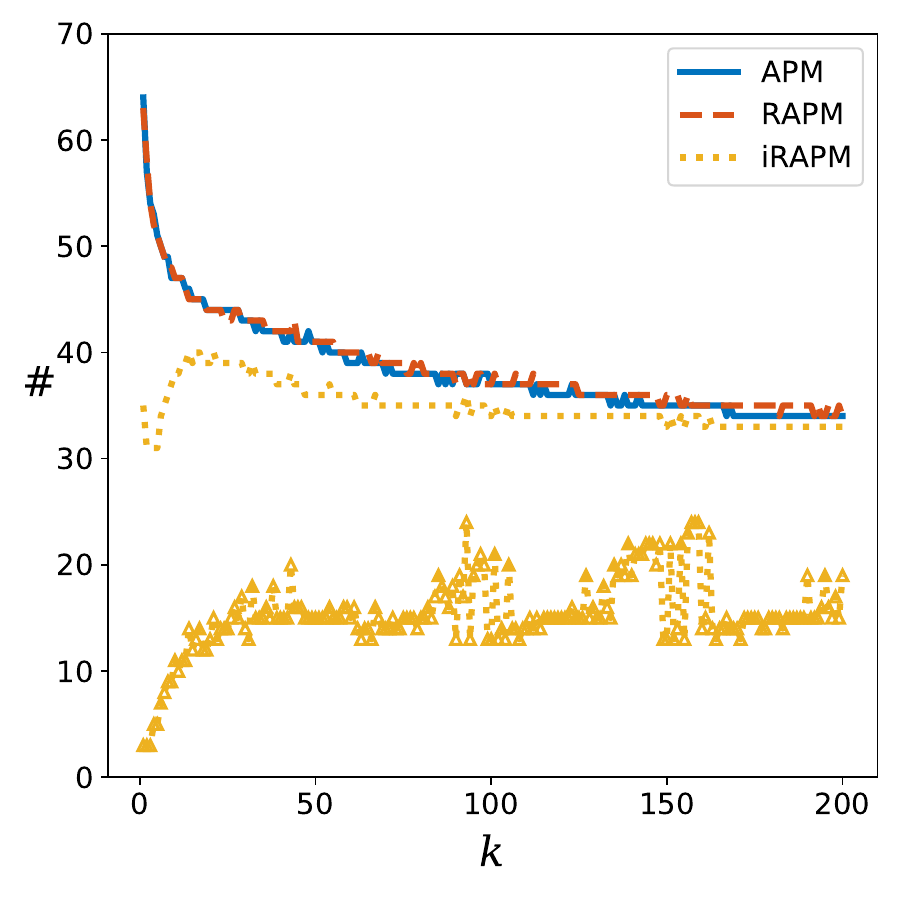}
    \vspace{-\baselineskip}
    \caption{Number of  inner iterations $\bar \ell_k$ along outer iterations and number of singular values that satisfy the criterion \eqref{eq:Lanczos_error_bound_v2}. Gaussian matrices (left) and truncated {\tt lake} (right).}
    \label{fig:krylov}
\end{figure}

%%%%%%%%%%%%%%%
% Conclusions %
%%%%%%%%%%%%%%%
\section{Conclusions and perspectives}

\REV{We have introduced an inexact regularized alternating projection method for computing a point in the intersection of two nonconvex sets. Inexact projections on one of the two sets are allowed. The convergence properties of the methods have been proved in a general setting and then specialized to affine rank minimization problems. For the latter class of problems the inexact projection on the rank level sets amounts to compute an approximate r-truncated SVD. Suitable  accuracy requirements and corresponding stopping criteria for the Krylov inner solver have been proposed. Numerical results on matrix completion problems showed that the inexact method is  in general less expensive than the exact counterpart.  
We observed that the savings in computational cost are concentrated  in the first phase of the convergence process. Then, it could be of interest to devise an adaptive strategy for choosing the parameter $\zeta$ that rules the inexactness level. A further extension is to devise adaptive strategies for the rank selection when the rank is unknown.}

\backmatter

% \bmhead{Supplementary information}

% If your article has accompanying supplementary file/s please state so here. 

% Authors reporting data from electrophoretic gels and blots should supply the full unprocessed scans for key as part of their Supplementary information. This may be requested by the editorial team/s if it is missing.

% Please refer to Journal-level guidance for any specific requirements.

\bmhead{Acknowledgements}

The authors are all members of the INdAM research group GNCS, which is kindly
acknowledged.

\bmhead{Author Contributions} All authors contributed to the study conception and design, and the theoretical content of the paper. Code implementation, data collection and analysis were performed by MS.
%The first draft of the manuscript was written by SB and all authors commented on previous versions of the manuscript.
All authors read and approved the final manuscript.

\bmhead{Funding} The research that led to the present paper was partially supported by INDAM-GNCS through Progetti di Ricerca 2025 (CUP\_E53C24001950001) and by PNRR - Missione 4 Istruzione e Ricerca - Componente C2 Investimento 1.1, Fondo per il Programma Nazionale di Ricerca e Progetti di Rilevante Interesse Nazionale (PRIN) funded by the European Commission under the NextGeneration EU programme, project ``Inverse problems in PDE: theoretical and numerical analysis'', code: 2022B32J5C, MUR D.D. financing decree n. 973 of 30th June 2023 (CUP B53D23009200006), project ``Advanced optimization METhods for automated central veIn Sign detection in multiple sclerosis from magneTic resonAnce imaging (AMETISTA)'',  code: P2022J9SNP, MUR D.D. financing decree n. 1379 of 1st September 2023 (CUP E53D23017980001), project ``Numerical Optimization with Adaptive Accuracy and Applications to Machine Learning'',  code: 2022N3ZNAX MUR D.D. financing decree n. 973 of 30th June 2023 (CUP B53D23012670006), and by Partenariato esteso FAIR ``Future Artificial Intelligence Research'' SPOKE 1 Human-Centered AI. Obiettivo 4, Project ``Mathematical and Physical approaches to innovative Machine Learning technologies (MaPLe)''.

\bmhead{Data availability statement}

The \verb|lake| image can be downloaded from \url{http://www.imageprocessingplace.com}

\section*{Declarations}

% Some journals require declarations to be submitted in a standardised format. Please check the Instructions for Authors of the journal to which you are submitting to see if you need to complete this section. If yes, your manuscript must contain the following sections under the heading `Declarations':

% \begin{itemize}
% \item Conflict of interest/Competing interests (check journal-specific guidelines for which heading to use)
% \item Ethics approval and consent to participate
% \item Consent for publication
% \item Materials availability
% \end{itemize}

% \noindent
% If any of the sections are not relevant to your manuscript, please include the heading and write `Not applicable' for that section.

\bmhead{Conflict of interest} All authors certify that they have no affiliations with or involvement in any organization or entity with any financial interest or non-financial interest in the subject matter or materials discussed in this manuscript. The authors have no competing interests to declare that are relevant to the content of this article.
\bmhead{Ethics approval} The authors declare that research ethics approval was not required for this study.
\bmhead{Informed consent} The authors declare that informed consents were not required for this study.
\bmhead{Open Access} This article is licensed under a Creative Commons Attribution 4.0 International License, which permits use, sharing, adaptation, distribution and reproduction in any medium or format, as long as you give appropriate credit to the original author(s) and the source, provide a link to the Creative Commons licence, and indicate if changes were made. The images or other third party material in this article are included in the article’s Creative Commons licence, unless indicated otherwise in a credit line to the material. If material is not included in the article’s Creative Commons licence and your intended use is not permitted by statutory regulation or exceeds the permitted use, you will need to obtain permission directly from the copyright holder. To view a copy of this licence, visit \url{http://creativecommons.org/licenses/by/4.0/.}

\begin{appendices}

\section{Additional proofs}

\subsection{Proof of Lemma~\ref{lem:stopping}}\label{appendix:A1}
If $y^k \in  \mathcal{P}_{B}(\yck)$,  one can always set $w^{\ell} = y^k$ and $c^{\ell} = \|y^k-\yck\|^2$ \REV{for any $\ell\in\N$}, so that conditions \eqref{eq:cond_1}-\eqref{eq:cond_3} and \eqref{eq:stopping} are trivially satisfied for all $\ell\in\N$ with $\hat{y}^{k+1}=y^k $. Then, without loss of generality, we assume that $y^k\notin \mathcal{P}_{B}(\yck)$. Hence, there holds $\norm{\hat{y}^{k+1} - \yck}^2 - \norm{y^k - \yck}^2 < 0$. 
First, we prove that there exists $\ell_\epsilon \in \N$ such that
\begin{equation}\label{eq:proof_step_1}
     \zeta ({c^{\ell}} - \norm{y^k - \yck}^2)-(\norm{\hat{y}^{k+1} - \yck}^2 - \norm{y^k - \yck}^2)>0  \quad \forall \, \ell \ge \ell_\epsilon.
\end{equation}
Indeed, by \eqref{eq:cond_2}, we have
\[
\forall \ \varepsilon > 0, \exists \ \ell_\varepsilon > 0 \, : \, c^\ell - \norm{\hat{y}^{k+1} - \yck}^2 > - \varepsilon, \quad \forall \ \ell \ge \ell_\varepsilon.
\]
Setting $\varepsilon = \frac{1-\zeta}{\zeta} (\norm{y^k-\yck}^2 - \norm{\hat{y}^{k+1} - \yck}^2)>0$, we have that there exists $\ell_{\epsilon} \in \N$ such that
\begin{equation}
    \zeta(c^\ell - \norm{\hat{y}^{k+1} - \yck}^2)  > - (1-\zeta) (\norm{y^k-\yck}^2 - \norm{\hat{y}^{k+1} - \yck}^2) \nonumber
\end{equation}
that yields \eqref{eq:proof_step_1}.
Now, we show that there exists $\hat \ell \geq   \ell_\epsilon$ such that \eqref{eq:stopping} holds. \REV{Indeed, it holds
\begin{equation}\label{eq:technical2}
        \|w^{\ell}-\yck\|^2=\|w^{\ell}-\hat{y}^{k+1}\|^2+\|\hat{y}^{k+1}-\yck\|^2+2\langle w^{\ell}-\hat{y}^{k+1},\hat{y}^{k+1}-\yck\rangle,
\end{equation}
and we observe} that there exists $\hat{\ell}\geq \ell_\epsilon$ such that, for all $\ell\geq \hat{\ell}$, it holds
\begin{align}\label{eq:technical3}
    \|w^{\ell}-\hat{y}^{k+1}\|^2+2\langle w^{\ell}-\hat{y}^{k+1},\hat{y}^{k+1}-\yck\rangle&\leq \zeta(c^{\hat{\ell}}-\|y^k-\yck\|^2)-\left(\|\hat{y}^{k+1}-\yck\|^2-\|y^k-\yck\|^2\right)\nonumber\\
    &\leq \zeta(c^{\ell}-\|y^k-\yck\|^2)-\left(\|\hat{y}^{k+1}-\yck\|^2-\|y^k-\yck\|^2\right),
\end{align}
{where the former inequality holds given that the left-hand side is converging to zero (see \eqref{eq:cond_3}) and the right-hand side is positive (see \eqref{eq:proof_step_1}), whereas the latter inequality is a consequence of \eqref{eq:cond_1}-\eqref{eq:cond_2} that ensures $c^{\hat \ell}< c^{\ell}$.} {By plugging \eqref{eq:technical3} into \eqref{eq:technical2}}  we obtain \eqref{eq:stopping}. Applying \eqref{eq:cond_1} to \eqref{eq:stopping} yields
\begin{equation*}
    \|w^{{\ell}}-\yck\|^2\leq \zeta \|\hat{y}^{k+1}-\yck\|^2+(1-\zeta)\|y^k-\yck\|^2,
\end{equation*}
which is equivalent to condition \eqref{eq:inexact} with $y^{k+1}=w^{{\ell}}$.

\subsection{Proof of Lemma~\ref{lem:properties}}\label{appendix:A2}
{\it Item (i).} We recall that $Q(y^{k+1})\leq 0$ due to equations \eqref{eq:sign-Q} and \eqref{eq:inexact}. Then, from the definition of $\{d_k\}_{k\in\N}$ in \eqref{eq:dk}, we obtain the thesis.

{\it Item (ii).} We observe that the inexactness condition \eqref{eq:inexact} can be reformulated by means of the following equivalences
\begin{align}
    &Q^{k+1}(y^{k+1}) \le \zeta Q^{k+1}(\hat{y}^{k+1}) \Leftrightarrow  Q^{k+1}(y^{k+1}) \leq Q^{k+1}(\hat{y}^{k+1}) - \frac{1-\zeta}{\zeta} Q^{k+1}(y^{k+1})\nonumber\\
    & \Leftrightarrow  \frac{\|y^{k+1}-y^k\|^2}{2\mu_k}+L(x^{k+1},y^{k+1})\leq \frac{\|\hat{y}^{k+1}-y^k\|^2}{2\mu_k}+L(x^{k+1},\hat{y}^{k+1})- \frac{1-\zeta}{\zeta} Q^{k+1}(y^{k+1})\label{eq:inexact_reformulation}.
\end{align}
Since $\hat{y}^{k+1}\in\mathcal{P}_B((y^k+\mu_kx^{k+1})/(1+\mu_k))$, we have $\hat{y}^{k+1}\in \prox_{\mu_k L(x^{k+1},\cdot)}(y^k)$, which yields
\begin{equation}\label{eq:projB}
    \frac{1}{2\mu_k}\|\hat{y}^{k+1}-y^k\|^2+L(x^{k+1},\hat{y}^{k+1})\leq L(x^{k+1},y^k).
\end{equation}
By applying the previous inequality to \eqref{eq:inexact_reformulation}, we obtain
\begin{equation}\label{ineq:interm}
    \frac{1}{2\mu_k}\|y^{k+1}-y^k\|^2+L(x^{k+1},y^{k+1})\leq L(x^{k+1},y^k)- \frac{1-\zeta}{\zeta} Q^{k+1}(y^{k+1}).
\end{equation}
Writing $Q^{k+1}(y^{k+1})$ as in \eqref{eq:Q} and rearranging the terms  of \eqref{ineq:interm} yields
\begin{equation*}
    \frac{1}{2\mu_k}\left(1+\frac{1}{\zeta}\right)\|y^{k+1}-y^k\|^2-Q^{k+1}(y^{k+1})+\left(1+\frac{1}{\zeta}\right)L(x^{k+1},y^{k+1})\leq \left(1+\frac{1}{\zeta}\right)L(x^{k+1},y^k).
\end{equation*}
By dividing both terms of the previous inequality  by a factor $1+1/\zeta$ and summing the term $\frac{1}{2\lambda_k}\|x^{k+1}-x^k\|^2$, we get
\begin{align*}
    \frac{1}{2\mu_k}\|y^{k+1}-y^k\|^2&+\frac{1}{2\lambda_k}\|x^{k+1}-x^k\|^2-\left(\frac{\zeta}{\zeta+1}\right)Q^{k+1}(y^{k+1})+L(x^{k+1},y^{k+1})\\
    &\leq \frac{1}{2\lambda_k}\|x^{k+1}-x^k\|^2+L(x^{k+1},y^k)
    \leq L(x^k,y^k),
\end{align*}
where the second inequality follows from $x^{k+1}\in\prox_{\lambda_k L(\cdot,y^k)}(x^k)$ and the definition of proximal operator.
By observing that $\lambda_k\leq \lambda_+$ and $\mu_k\leq \mu_+$, we arrive to
\begin{equation}\label{eq:decrease}
    \frac{1}{2\mu_+}\|y^{k+1}-y^k\|^2+\frac{1}{2\lambda_+}\|x^{k+1}-x^k\|^2-\left(\frac{\zeta}{\zeta+1}\right)Q^{k+1}(y^{k+1})+L(x^{k+1},y^{k+1})\leq L(x^{k},y^k).
\end{equation}
Then the thesis of item (ii) follows from \eqref{eq:decrease} by setting $a=\min\{1/(2\mu_+),1/(2\lambda_{+}),\zeta/(\zeta+1)\}$ and recalling the definition of $d_k$ in \eqref{eq:dk}.

{\it Item (iii).} The left-hand side of \REV{\eqref{eq:limit3}}  follows immediately from \eqref{eq:inexact_reformulation} by recalling the definition of $F$ in \eqref{eq:merit} and $\{\rho_k\}_{k\in\N}$  in \eqref{eq:rhok}. Regarding the right-hand side, \REV{we sum $\gamma_k$ given in \eqref{eq:rk} to both sides of \eqref{eq:projB}, thus obtaining}
\begin{align*}
    F(x^{k+1},\hat{y}^{k+1},\rho_{k+1})&\leq \frac{1}{2\mu_k}\|\hat{y}^{k+1}-y^k\|^2+L(x^{k+1},\hat{y}^{k+1})+\gamma_k\\
    &\leq L(x^{k+1},y^k)+\gamma_k\leq L(x^{k+1},y^k)+\frac{1}{2\lambda_k}\|x^{k+1}-x^{k}\|^2+\gamma_k\leq L(x^k,y^k)+\gamma_k,
\end{align*}
where \REV{the first inequality follows from $Q^{k+1}(\hat{y}^{k+1})=\min_{y\in\R^n}Q^{k+1}(y)\leq Q^{k+1}(y^{k+1})$}, and the last one holds by $x^{k+1}\in\prox_{\lambda_k L(\cdot,y^k)}(x^k)$ and the definition of proximal operator. We observe that the sequence $\{\gamma_k\}_{k\in\N}$ is non-negative due to \eqref{eq:sign-Q}. Furthermore, by summing \eqref{eq:decrease} for $k=0,\ldots,K$, it follows that
\begin{equation*}
    \frac{1}{2\mu_+}\sum_{k=0}^{K}\|y^{k+1}-y^k\|^2+\frac{1}{2\lambda_+}\sum_{k=0}^{K}\|x^{k+1}-x^k\|^2+\left(\frac{\zeta}{\zeta+1}\right)\sum_{k=0}^{K}-Q^{k+1}(y^{k+1})\leq L(x^{0},y^0)-L(x^{K+1},y^{K+1}).
\end{equation*}
By taking the limit over $K\rightarrow \infty$ and recalling the definition of $a$, we get
\begin{equation}\label{eq:xyQ_summable}
    {a} \sum_{k=0}^\infty\|x^{k+1}-x^k\|^2+\|y^{k+1}-y^k\|^2-Q^{k+1}(y^{k+1})<\infty.
\end{equation}
As a result, the sequence $\{ -Q^{k+1}(y^{k+1}) \}_{k\in\N}$ is summable and, thanks to the inexactness condition \eqref{eq:inexact}, we have that $\{ -Q^{k+1}(\hat{y}^{k+1}) \}_{k\in\N}$ is also summable. Then, \REV{\eqref{eq:limit2} holds due to condition \eqref{eq:bound}, and $\{\gamma_k\}_{k\in\N}$ is summable, which yields $\lim_{k\rightarrow \infty}\gamma_k = 0$.}
Furthermore, it holds $\lim_{k\rightarrow \infty}\|y^{k+1}-y^k\|=0$ by \eqref{eq:xyQ_summable}. Then, by applying the triangular inequality, we conclude that $\lim_{k\rightarrow \infty}\|\hat{y}^{k+1}-y^k\|=0$ and consequently $\lim_{k\rightarrow \infty}\rho_k=0$.  Hence, the thesis of item (iii) follows.

{\it Item (iv)}. Since the subdifferential of the sum of two functions, one of which is continuously differentiable, is given by the sum of the two subdifferentials \cite[Exercise 8.8]{Rockafellar-Wets-1998}, it is possible to write the following equivalences
\begin{align*}
    x^{k+1} \in \prox_{\lambda_{k} L(\cdot,y^{k})} (x^{k}) 
    &\iif 0 \in (x^{k+1} - y^{k}) + \partial \iota_A(x^{k+1}) + \frac{1}{\lambda_{k}} (x^{k+1} - x^{k}), \nonumber \\
    \hat y^{k+1} \in \prox_{\mu_{k} L(x^{k+1},\cdot)} 
    &\iif 0 \in -(x^{k+1} - \hat y^{k+1}) + \partial \iota_B(\hat y^{k+1}) + \frac{1}{\mu_{k}} (\hat y^{k+1} - y^{k}).
\end{align*}
Hence there exist $w_1^{k+1} \in \partial \iota_A(x^{k+1})$ and $w_2^{k+1} \in \partial \iota_B(\hat y^{k+1})$ 
such that
\begin{equation*}
    w_1^{k+1} = -(x^{k+1} - y^{k}) - \frac{1}{\lambda_{k}} (x^{k+1} - x^{k}), \quad
    w_2^{k+1} = (x^{k+1} - \hat y^{k+1}) - \frac{1}{\mu_{k}} (\hat y^{k+1} - y^{k}).
\end{equation*}
The subdifferential of $L$ at any point $(x,y)$ writes as $\partial L(x,y) = \{ x - y + \partial \iota_A(x) \} \times \{ y - x + \partial \iota_B(y) \}$, see e.g. \cite[Proposition 2.1]{attouch2010proximal}. Then, if we define the vectors
\begin{align*}
    v_1^{k+1} &= x^{k+1} - \hat y^{k+1} + w_1^{k+1}=-\frac{1}{\lambda_k}(x^{k+1}-x^k)-(\hat{y}^{k+1}-y^k),  \\
    v_2^{k+1} &= \hat y^{k+1} - x^{k+1} + w_2^{k+1}=-\frac{1}{\mu_k}(\hat{y}^{k+1}-y^k), 
\end{align*}
the vector $v^{k+1} = (v_1^{k+1}, v_2^{k+1})$ belongs to the subdifferential of $L$ at $(x^{k+1},\hat{y}^{k+1})$. By setting $\kappa_- = \min\{\lambda_{-},\mu_{-}\}$, the norm of $v^{k+1}$ can be upper bounded by proceeding as follows
\begin{align}
    \|v^{k+1}\|&= \left\| \left( -\frac{1}{\lambda_{k}} (x^{k+1} - x^{k}) - ( \hat y^{k+1} - y^{k}), - \frac{1}{\mu_{k}} (\hat y^{k+1} - y^{k}) \right) \right\|\nonumber\\
    &= \left\| 
    \left(
    -\frac{1}{\lambda_{k}} (x^{k+1} - x^{k}),
                        - \frac{1}{\mu_{k}} (y^{k+1}- y^{k})
                    \right) + 
                    \left(
                        - ( \hat y^{k+1} - y^{k}),
                        - \frac{1}{\mu_{k}} (\hat y^{k+1} - y^{k+1})
                    \right) \right\|\nonumber\\
    &\le \frac{1}{\kappa_-} \left\|
                    \left(
                        x^{k+1} - x^{k},
                        y^{k+1} - y^{k}
                    \right) \right\| +\left\|
                    \left(
                        -(\hat y^{k+1} - y^{k+1}),
                        - \frac{1}{\mu_{k}} (\hat y^{k+1} - y^{k+1})
                    \right) \right\| \nonumber\\
    & + \left\|- \left( y^{k+1} - y^{k}, 0 \right)\right\|\leq \left(1+\frac{1}{\kappa_-}\right)\|(x^{k+1}-x^k,y^{k+1}-y^k)\|+\sqrt{1+\frac{1}{\kappa_-^2}}\|\hat{y}^{k+1}-y^{k+1}\|\nonumber\\
    &\leq \left(1+\frac{1}{\kappa_-}\right)\|(x^{k+1}-x^k,y^{k+1}-y^k)\|+\sqrt{-\left(1+\frac{1}{\kappa_-^2}\right)\left(\frac{1-\zeta}{\zeta}\right)Q^{k+1}(y^{k+1})},\label{eq:subdifferential_crucial}
\end{align}
where the last inequality follows from condition \eqref{eq:bound}. Finally, we observe that the subdifferential of the function $F$ defined in \eqref{eq:merit} can be written as $ \partial F(x,y,\rho) = \partial L(x,y) \times \{\rho\}$. Hence, since the vector $v^{k+1}\in\R^n$ belongs to the subdifferential $\partial L(x^{k+1},\hat{y}^{k+1})$, we have that $(v^{k+1},\rho^{k+1})\in\partial F(x^{k+1},\hat{y}^{k+1},\rho^{k+1})$, and the norm of such subgradient can be upper bounded as 
\begin{align*}
    &\|(v^{k+1},\rho^{k+1})\|\leq \|v^{k+1}\|+|\rho^{k+1}|\leq \left(1+\frac{1}{\kappa_-}\right)\|(x^{k+1}-x^k,y^{k+1}-y^k)\|\\
    &+\sqrt{-\left(1+\frac{1}{\kappa_-^2}\right)\left(\frac{1-\zeta}{\zeta}\right)Q^{k+1}(y^{k+1})}
    +\sqrt{\frac{1}{\mu_k}}\|\hat{y}^{k+1}-y^k\|+\sqrt{-2\left(\frac{1-\zeta}{\zeta}\right)Q^{k+1}(y^{k+1})}\\
    &\leq \left(1+\frac{1}{\kappa_-}+\frac{1}{\sqrt{\kappa_-}}\right)\|(x^{k+1}-x^k,y^{k+1}-y^k)\|\\
    &+\sqrt{\frac{1-\zeta}{\zeta}}\left(\sqrt{\frac{1}{\kappa_-}}+\sqrt{1+\frac{1}{\kappa_-^2}}+\sqrt{2}\right)\sqrt{-Q^{k+1}(y^{k+1})},
\end{align*}
where the second inequality is due to \eqref{eq:subdifferential_crucial} and the definition of $\{\rho_k\}_{k\in\N}$ in \eqref{eq:rhok}, \REV{whereas the third one follows by employing the triangular inequality} and the inexactness condition \eqref{eq:bound}. {Then, the thesis of item (iv) is obtained by applying \eqref{eq:ineq_dk} to the previous inequality.}




\end{appendices}


\bibliography{bibliography}% common bib file

\end{document}